\documentclass[smallextended,referee,envcountsect]{svjour3}
\smartqed
\usepackage{graphicx}
\usepackage{epstopdf}
\usepackage{amsopn}
\usepackage{booktabs}
\usepackage{multirow}
\usepackage{makecell}
\usepackage{placeins}

\usepackage{amsmath,amssymb,amsfonts}%
\usepackage{algorithm}
\usepackage{algpseudocode}

\usepackage{caption}

\newcommand{\dd}{\,{\rm d}}
\newcommand{\dual}[1]{\left\langle {#1} \right\rangle}

\usepackage{tikz}
\usetikzlibrary{positioning}

\newcommand{\bs}{\boldsymbol}

\usepackage{enumitem}
\usepackage{comment}

\newcommand{\RV}[1]{#1}
\definecolor{atomictangerine}{rgb}{1.0, 0.6, 0.4}

\usepackage{titlesec}

\titlespacing*{\section}
  {0pt}      
  {4ex plus 0.5ex minus 0.2ex}  
  {1.5ex plus 0.2ex}              

\titlespacing*{\subsection}
  {0pt}
  {2ex plus 0.4ex minus 0.2ex}
  {1ex plus 0.2ex}

\titlespacing*{\subsubsection}
  {0pt}
  {1.2ex plus 0.3ex minus 0.2ex}
  {0.75ex plus 0.2ex}
  
\journalname{JOTA}

\begin{document}

\title{HNAG$^{++}$: An Accelerated Gradient Method with a Refined Asymptotic Rate for Strongly Convex Optimization}
\titlerunning{HNAG$^{++}$: Accelerated Gradient Method with Refined Rate}


\author{Long Chen         \and
        Zeyi Xu
}
\authorrunning{L. Chen \and Z. Xu} 

\institute{Long Chen \at
             Department of Mathematics, University of California, Irvine, CA 92697, USA.\\
              lchen7@uci.edu
           \and
              Zeyi Xu, Corresponding author  \at
              Department of Mathematics, University of California, Irvine, CA 92697, USA.\\
              zeyix1@uci.edu
}

\date{Received: date / Accepted: date}

\maketitle

\begin{abstract}
Two accelerated first-order methods, HNAG$^+$ and HNAG$^{++}$, are introduced for smooth strongly convex optimization. They are derived from the \RV{Hessian-driven Nesterov Accelerated Gradient} (HNAG) flow by optimizing the coercivity of shifted Lyapunov functions. \RV{Let $\kappa=L/\mu$, where $\mu$ is the strong-convexity constant and $L$ is the gradient Lipschitz constant}. HNAG$^+$ attains the optimal global rate $1-2/\sqrt{\kappa}$, matching the information-theoretic lower bound. For functions with local asymptotic symmetry at the minimizer, HNAG$^{++}$ attains the asymptotic rate $1-2\sqrt{2/\kappa}$. This matches the best known asymptotic rate under $\mathcal C^2$ regularity, while applying to a broader function class. Numerical experiments confirm the predicted rates and \RV{show favorable performance} against existing accelerated schemes.\end{abstract}

\section{Introduction}
We consider the unconstrained smooth convex optimization problem
\begin{equation*}
\min_{x\in\mathbb{R}^d} f(x),
\end{equation*}
where $f$ belongs to the class $\mathcal S_{\mu,L}$ of continuously differentiable, $\mu$-strongly convex functions with $L$-Lipschitz continuous gradients. Strong convexity guarantees the existence
of a unique global minimizer $x^{\star}$. We study first-order methods, which
rely solely on gradient information, and generate iterates 
$\{x_k\}$ converging to $x^{\star}$.

Convergence is measured by a nonnegative error quantity $E_k$, such as $\|x_k-x^\star\|^2$, $f(x_k)-f(x^\star)$, or a Lyapunov function with comparable scaling. We say that the sequence $\{E_k\}$
converges linearly if
$$
E_k \le C r^k, \qquad k \ge 1,
$$
for some constants $C>0$ and $r\in(0,1)$.
More precisely, $E_k$ is said to converge \emph{R-linearly} with asymptotic rate
$r$ if
$$
\limsup_{k\to\infty} E_k^{1/k} \le r.
$$

The condition number of $f$ is $\kappa = L/\mu$, and we denote its reciprocal by
$\rho = 1/\kappa = \mu/L$. 
We focus on ill-conditioned problems with $\kappa \gg 1$ (equivalently,
$\rho \ll 1$) and express convergence rates to leading order in $\rho$. For
example,
$$
r = (1 + c\sqrt{\rho})^{-1}
\;\approx\;
1 - c\sqrt{\rho},
\qquad \rho \ll 1.
$$
To reach an accuracy level $r^k \le \varepsilon$, one requires on the order of
$\ln(1/\varepsilon)\,\sqrt{\kappa}/c$ iterations. Consequently, increasing the
constant $c$ directly reduces the iteration complexity; doubling $c$
approximately halves the iteration count and leads to substantial computational
savings.

\subsection{Related Work}
While Gradient Descent (GD) converges linearly at a rate $(\kappa-1)/(\kappa+1)\approx 1-2/\kappa$ with the optimal step size $2/(L+\mu)$ for smooth strongly convex functions, its performance degrades rapidly as $\kappa$ increases.

Momentum-based acceleration overcomes this bottleneck. Although the classic Heavy-Ball (HB) method~\cite{polyak1964some} lacks global convergence guarantees for general strongly convex objectives~\cite{lessard2016analysis,goujaud2023provable}, global and accelerated convergence is established by Nesterov’s Accelerated Gradient (NAG)~\cite{nesterov1983method,nesterov2003introductory} and recent variants like AOR-HB~\cite{wei2024accelerated}, which typically achieve a rate $1-1/\sqrt{\kappa}$.

More sophisticated schemes push the theoretical limits. The Triple Momentum (TM) method~\cite{van2017fastest} and the Information-Theoretic Exact Method (ITEM)~\cite{taylor2021optimal_gradient_method} attain the optimal first-order rate $1-2/\sqrt{\kappa}$. Under stronger regularity assumptions, the Optimized Gradient Method (OGM)~\cite{kim2016optimized_first_order_methods} for quadratics and $C^2$-Momentum (C2M)~\cite{vanScoy2025fastest_firstorder_twiceC2} for $f\in \mathcal C^2$, i.e., the Hessian of $f$ is continuous, achieve the sharper asymptotic rate $1-2\sqrt{2/\kappa}$.

\subsection{Approach and Contribution}
Optimization algorithms can be interpreted as discretizations of continuous-time dynamics; see, for example,~\cite{su2016differential,wibisono2016variational,AttoucChbaniFadiliRiahi2022First-order,shi2022understanding,luoDifferentialEquationSolvers2021a,Chen;Hao;Wei:2025Accelerated}. However, compared with the Performance Estimation Problem (PEP) framework~\cite{DroriTeboulle2014} and the Integral Quadratic Constraint (IQC) framework~\cite{lessard2016analysis,taylor2018lyapunov}, the ODE approach has not yet achieved the same level of sharpness in convergence-rate analysis.

The present work bridges this gap. We study acceleration through the Hessian-driven Nesterov Accelerated Gradient (HNAG) flow introduced in \cite{chen2019orderoptimizationmethodsbased}. 
The main tool is a Lyapunov analysis with a coercivity enhancement mechanism based on primal and dual shifts.

Compared with PEP~\cite{DroriTeboulle2014} and IQC~\cite{lessard2016analysis,taylor2018lyapunov}, our approach is more explicit and geometric. Compared with existing accelerated schemes such as TM and C2M, the main contribution is not the rate alone, but the design principle and the Lyapunov analysis. More precisely, our main contributions are:

\begin{enumerate}[leftmargin=16pt,itemsep=4pt]

\item \textbf{HNAG$^{+}$: optimal rate from a coercivity-based design principle.} \\
We show that HNAG$^{+}$ reaches the optimal rate $1-2/\sqrt{\kappa}$ on $\mathcal S_{\mu,L}$. The resulting scheme is closely related to TM~\cite{van2017fastest}; indeed, TM appears as a special parameter choice within our framework. The contribution here is the continuous-time derivation, the coercivity-based interpretation, and a Lyapunov analysis with fixed, closed-form parameters.

\item \textbf{HNAG$^{++}$: faster asymptotic rate on a broader function class.} \\
We introduce the LAS class $\mathcal S_{\mu,L}^{\rm LAS}(x^\star)$ in Section~\ref{sec:LAS}, where the Bregman asymmetry is negligible relative to the local Bregman energy near $x^\star$. For this class, HNAG$^{++}$ attains the faster asymptotic rate $1-2\sqrt{2/\kappa}$ with step size $\sqrt{2/\kappa}$. This matches the best known asymptotic rate proved for C2M~\cite{vanScoy2025fastest_firstorder_twiceC2}, but uses simpler parameters and a broader function class as $\mathcal C^2\cap \mathcal S_{\mu,L} \subset \mathcal S_{\mu,L}^{\rm LAS}(x^\star)$.
\end{enumerate}

Table~\ref{tab:compare_methods} summarizes the rates of several accelerated first-order methods. 

\begin{table}[htbp]
\centering
\renewcommand{\arraystretch}{1.3}
\begin{tabular}{@{}l l l c@{}}
  \toprule
  Method & Function class & Reference & Rate (leading order)\\
  \midrule
  NAG / HNAG
  & $\mathcal S_{\mu,L}$
  & \cite{nesterov2003introductory,chen2019orderoptimizationmethodsbased}, \textbf{Sec.~2}
  & $1 - 1/\sqrt{\kappa}$ \\

  NAG / HNAG
  & $\mathcal S_{\mu, L}^{\rm LAS}(x^{\star})$
  & \textbf{Sec.~5}
  & asy $1 - 2/\sqrt{\kappa}$ \\

  \textbf{HNAG$^{+}$}
  & $\mathcal S_{\mu,L}$
  & \textbf{Sec.~4}
  & $\mathbf{1 - 2/\sqrt{\kappa}}$ \\

  TM / ITEM / GAG
  & $\mathcal S_{\mu,L}$
  & \cite{van2017fastest,taylor2021optimal_gradient_method,Wu2024generalizedAGD}
  & $1 - 2/\sqrt{\kappa}$ \\

  C2M
  & $\mathcal C^2\cap \mathcal S_{\mu,L}$
  & \cite{vanScoy2025fastest_firstorder_twiceC2}
  & asy $1 - 2\sqrt{2/\kappa}$ \\

  \midrule
  \textbf{HNAG$^{++}$}
  & $\mathcal S_{\mu, L}^{\rm LAS}(x^{\star})$
  & \textbf{Sec.~5}
  & \textbf{asy} $\mathbf{1 - 2\sqrt{2/\kappa}}$ \\
  \bottomrule
\end{tabular}
\caption{Convergence rates for accelerated first-order methods under different smoothness assumptions: $\mathcal C^2\cap\mathcal S_{\mu,L}\subset \mathcal S_{\mu, L}^{\rm LAS}(x^{\star})\subset \mathcal S_{\mu,L}$. Here ``asy'' denotes an asymptotic rate; the other entries are global linear rates.}
\vskip -12pt
\label{tab:compare_methods}
\end{table}

\subsection{Preliminaries}
Let $f: \mathbb{R}^d \to \mathbb{R}$ be differentiable. We use $\langle \cdot, \cdot \rangle$ for the inner product of $\mathbb R^d$.

The Bregman divergence of $f$ between $x, y \in \mathbb{R}^d$ is 
\begin{equation*}
D_f(y,x) := f(y) - f(x) - \langle \nabla f(x),\, y - x \rangle .
\end{equation*}
In general, $D_f$ is not symmetric, i.e., $D_f(y,x) \neq D_f(x,y)$.  
Its symmetrization, called the symmetrized Bregman divergence, is  
\begin{equation}\label{eq:symBre}
D_f(y,x) + D_f(x,y)
= \langle \nabla f(y) - \nabla f(x),\, y - x \rangle .
\end{equation}
Define the difference of Bregman divergences
\begin{equation*}
\Delta_f(x,y) := D_f(y,x) - D_f(x,y).
\end{equation*}
We can rewrite the symmetric Bregman divergence via the asymmetry:
\begin{equation}\label{eq:Brelation}
\langle \nabla f(x) - \nabla f(y),\, x - y \rangle 
= 2D_f(x,y) + \Delta_f(x,y).
\end{equation}
The term $\Delta_f(x,y)$ will be treated as a high-order perturbation in the asymptotic analysis. 

The following three-point Bregman identity holds~\cite{chen1993convergence}:
\begin{equation}\label{eq:Bregmanidentity}
\langle \nabla f(y) - \nabla f(x),\, y - z \rangle 
= D_f(y,x) + D_f(z,y) - D_f(z,x),
\end{equation}
which is a generalization of the classical identity of squares:
\begin{equation}\label{eq:squares}
\langle y-x,\, y-z \rangle
=
\frac12\bigl(
\|y-x\|^2 + \|z-y\|^2 - \|z-x\|^2
\bigr).
\end{equation}

The function $f$ is $\mu$-strongly convex if for some $\mu > 0$,
\begin{equation*}
D_f(y,x) \geq \frac{\mu}{2} \|y - x\|^2, 
\qquad \forall x, y \in \mathbb{R}^d.
\end{equation*}
It is $L$-smooth, for some $L > 0$, if its gradient is $L$-Lipschitz:
\begin{equation*}
\|\nabla f(y) - \nabla f(x)\| \leq L \|y - x\|, 
\qquad \forall x, y \in \mathbb{R}^d.
\end{equation*}

Let $\mathcal{S}_{\mu,L}$ be the class of all differentiable functions that are both $\mu$-strongly convex and $L$-smooth.  
For $f \in \mathcal{S}_{\mu,L}$, the Bregman divergence satisfies~\cite{nesterov2003introductory}
\begin{equation}\label{eq:muLsquares}
\frac{\mu}{2} \| x - y \|^2 
\leq D_f(x,y) 
\leq \frac{L}{2} \| x - y \|^2, 
\qquad \forall x, y \in \mathbb{R}^d,
\end{equation}
and, in terms of gradient differences,
\begin{equation}\label{eq:cocov}
\frac{1}{2L} \|\nabla f(x) - \nabla f(y)\|^2 
\leq D_f(x,y) 
\leq \frac{1}{2\mu} \|\nabla f(x) - \nabla f(y)\|^2, 
\qquad \forall x, y \in \mathbb{R}^d.
\end{equation}

Let $f^*$ be the convex conjugate of $f$.  
It is well known that if $f \in \mathcal{S}_{\mu,L}$, then $f^* \in \mathcal{S}_{1/L,\,1/\mu}$~\cite[Theorem~5.26]{beck2017first}.  

\subsection{Main idea}
\RV{For the HNAG$^+$ construction, we assume $L>\mu>0$ as the case $L=\mu$ is degenerate: one gradient step with step size $1/\mu$ reaches the minimizer.}
Given an $f\in \mathcal S_{\mu,L}$ with $L>\mu$, the shifted function
\begin{equation*}
f_{-\mu}(x) := f(x) - \frac{\mu}{2}\|x - x^{\star}\|^2
\end{equation*}
is convex and shares the same global minimum $x^{\star}$.  
Let $(f_{-\mu})^*$ be its convex conjugate.  
Since $f_{-\mu}$ is $(L-\mu)$-smooth, $(f_{-\mu})^*$ is $\tfrac{1}{L-\mu}$-strongly convex.  
Then the shift in the dual space,
\begin{equation*}
\bigl((f_{-\mu})^*\bigr)_{-\frac{1}{{L-\mu}}}(\chi) 
= (f_{-\mu})^*(\chi) - \frac{1}{2(L-\mu)}\|\chi\|^2 ,
\end{equation*}
is still convex.  
We use this {\it double shift} in the primal and dual spaces to construct a Lyapunov function. The primal shift $f_{-\mu}$ increases the coercivity, while the dual shift permits a larger step size.

\section{HNAG: Method and Convergence}

We recall the Hessian-driven Nesterov Accelerated Gradient (HNAG) method
introduced in~\cite{chen2019orderoptimizationmethodsbased}. We illustrate how accelerated convergence arises from a strong Lyapunov property and highlight the key algorithmic ingredients.

\subsection{Flow and Discretization}
The continuous-time HNAG flow is
\begin{equation}\label{eq:flow}
\left\{
\begin{aligned}
x' &= y - x - \beta \nabla f(x), \\
y' &= x - y - \tfrac{1}{\mu}\nabla f(x),
\end{aligned}
\right.
\end{equation}
where $\mu$ is the strong convexity parameter of $f$.
Let $\boldsymbol z=(x,y)$ and define
\begin{equation}\label{eq:G}
\mathcal G(\boldsymbol z)
=\bigl(y-x-\beta\nabla f(x),\; x-y-\tfrac{1}{\mu}\nabla f(x)\bigr)^{\intercal},
\end{equation}
so that \eqref{eq:flow} can be written compactly as $\boldsymbol z'=\mathcal G(\boldsymbol z)$.

An implicit--explicit (IMEX) discretization of~\eqref{eq:flow} yields
\begin{subequations}\label{eq:HNAG}
\begin{align}
\label{eq:HNAGx}
\frac{x_{k+1}-x_k}{\alpha} &= y_k - x_{k+1} - \beta \nabla f(x_k), \\
\label{eq:HNAGy}
\frac{y_{k+1}-y_k}{\alpha} &= x_{k+1} - y_{k+1} - \tfrac{1}{\mu}\nabla f(x_{k+1}),
\end{align}
\end{subequations}
where $x_{k+1}$ is updated first in \eqref{eq:HNAGx} and then used to compute $y_{k+1}$ in \eqref{eq:HNAGy}.

\vskip -10pt
\subsection{Algorithm}
Algorithm~\ref{alg:HNAG} gives an equivalent form for 
$$\alpha=\sqrt{\mu/L}, \quad \alpha\beta=1/L.$$ It uses one gradient evaluation per iteration and avoids division by the possibly small $\mu$. It follows from \eqref{eq:HNAG} by setting $v_{k+1}=\alpha y_k$, multiplying \eqref{eq:HNAGy} by $\alpha^2$ with an index shift to obtain line $3$, and multiplying \eqref{eq:HNAGx} by $\alpha$ to obtain line $4$.

\vskip -10pt
\begin{algorithm}
\caption{HNAG}
\label{alg:HNAG}
\linespread{1.3}\selectfont
\begin{algorithmic}[1]
\Require $x_0,v_0\in\mathbb R^d$, parameters $L\ge\mu>0$
\State $\alpha\gets\sqrt{\mu/L}$
\For{$k=0,1,2,\ldots$}
\State $v_{k+1}\gets \tfrac{1}{1+\alpha}\!\left(v_k+\alpha^2 x_k-\tfrac{1}{L}\nabla f(x_k)\right)$
\State $x_{k+1}\gets \tfrac{1}{1+\alpha}\!\left(x_k+v_{k+1}-\tfrac{1}{L}\nabla f(x_k)\right)$
\EndFor
\State \textbf{Return:} $x_{k+1}$ and ${y_{k}} = v_{k+1}/\alpha$
\end{algorithmic}
\end{algorithm}


\subsection{Strong Lyapunov Property}
We use the following Lyapunov notation throughout the paper. First, define
\begin{equation}\label{eq:lyapunov_strong_HNAG}
E(\boldsymbol z)=E(x,y)
:=D_f(x,x^\star)+\frac{\mu}{2}\|y-x^\star\|^2
=
f(x)-f(x^\star)+\frac{\mu}{2}\|y-x^\star\|^2,
\end{equation}
where \(D_f(x,x^\star)=f(x)-f(x^\star)\) since \(\nabla f(x^\star)=0\). Since
\(f\in\mathcal S_{\mu,L}\), we have \(E(\boldsymbol z)\ge0\), with equality if
and only if \(x=y=x^\star\).

For the \(\mu\)-shifted function \(f_{-\mu}\), we use a different font $\mathcal E(x,y)$.
For the dual-space shift, we add a tilde: $\widetilde E(x,y)$ or $\widetilde{\mathcal E}(x,y)$.
For partial shifts with \(\hat\mu\in[0,\mu]\), the shift is stated explicitly in
the notation; see \eqref{eq:lya-muhat}. We write \(\mathcal E^+(x,y)\) for the
rescaled Lyapunov function associated with HNAG$^+$; see
\eqref{eq:lyaderivation_scaled}.

\begin{lemma}\label{lem:cont_strong}
Let $\mathcal G(\boldsymbol z)$ be the vector field of the HNAG flow~\eqref{eq:flow}.
Then $E(\boldsymbol z)$ satisfies the following strong Lyapunov condition
\begin{equation}\label{eq:A-HNAG}
-\langle\nabla E(\boldsymbol z),\mathcal G(\boldsymbol z)\rangle
\;\ge\;
E(\boldsymbol z)+\beta\|\nabla f(x)\|^2+\frac{\mu}{2}\|x-y\|^2 .
\end{equation}
\end{lemma}

\begin{proof}
A direct computation yields
$$
\begin{aligned}
&\quad -\langle\nabla E(\boldsymbol z),\mathcal G(\boldsymbol z)\rangle \\ 
&=
\dual{\nabla f(x),x-x^\star}
+\beta\|\nabla f(x)\|^2
+\mu\|y-x^\star\|^2
-\mu\dual{y-x^\star,x-x^\star} \\
&=
E(\boldsymbol z)
+D_f(x^\star,x)
+\beta\|\nabla f(x)\|^2
+\frac{\mu}{2}\|x-y\|^2
-\frac{\mu}{2}\|x-x^\star\|^2,
\end{aligned}
$$
where we use identities of squares~\eqref{eq:squares} to expand the cross term. 
Using the $\mu$-convexity $D_f(x^\star,x)\ge \tfrac{\mu}{2}\|x-x^\star\|^2$ gives \eqref{eq:A-HNAG}.
$\Box$ \end{proof}

The analysis hinges on the {\it strong Lyapunov property} \cite{ChenLuo2021unified}
\begin{equation}\label{eq:strongLya}
-\langle\nabla E(\boldsymbol z),\mathcal G(\boldsymbol z)\rangle\;\ge\; c_L\,E(\boldsymbol z),
\end{equation}
where $c_L>0$ is analogous to the coercivity constant in PDE theory. Let $\boldsymbol z(t)$ be the solution of the ODE $\boldsymbol z'=\mathcal G(\boldsymbol z)$. By the chain rule and the strong Lyapunov property \eqref{eq:strongLya},
$$
\frac{\dd}{\dd t}E(\boldsymbol z(t))
= \langle \nabla E(\boldsymbol z(t)),\, \boldsymbol z'(t)\rangle
= \langle \nabla E(\boldsymbol z(t)),\,\mathcal G(\boldsymbol z(t))\rangle
\le -c_L\,E(\boldsymbol z(t)).
$$
Integrating this differential inequality, we obtain the exponential decay
$$
E(\boldsymbol z(t))\le E(\boldsymbol z(0))e^{-c_L t}, \qquad t\geq 0.
$$

Therefore, larger $c_L$ implies faster energy dissipation and faster convergence of a stable discretization.


\subsection{Convergence analysis}
As a conceptual benchmark, if the implicit Euler scheme $\boldsymbol z_{k+1}-\boldsymbol z_k=\alpha\mathcal G(\boldsymbol z_{k+1})$ were used,~\eqref{eq:strongLya} would directly imply linear convergence with rate $(1+c_L\alpha)^{-1}$. However, fully implicit schemes are infeasible. Some terms are evaluated using values at the current step, and this introduces a time lag which will be compensated by the convexity of the Lyapunov function.

We provide one-step convergence for \eqref{eq:HNAG} to illustrate the procedure. To keep the analysis clear, we focus on the leading order of $\alpha = \sqrt{\rho}$. 
\begin{lemma}\label{lem:disc_strong}
Let $(x_k,y_k)$ be the iterates generated by~\eqref{eq:HNAG} and $E$ defined by \eqref{eq:lyapunov_strong_HNAG}. Then
\begin{equation}\label{eq:onestep}
\begin{aligned}
(1+\alpha)E(\bs z_{k+1}) \;\le\;& E(\bs z_k)
+ \Bigl(\frac{\alpha^2}{2\mu}-\frac{\alpha\beta}{2}\Bigr)\|\nabla f(x_{k+1})\|^2
- \frac{\alpha\beta}{2}\|\nabla f(x_k)\|^2 \\
\qquad-\frac{\alpha \mu}{2}\|x_{k+1} - y_{k+1}\|^2 &+ \frac{\alpha\beta}{2}\|\nabla f(x_{k+1})-\nabla f(x_k)\|^2
- D_f(x_k,x_{k+1}).
\end{aligned}
\end{equation}
\end{lemma}
\begin{proof}
Expand the difference of $E$ by the definition of the Bregman divergence:
\begin{equation*}
E(\bs z_{k+1})-E(\bs z_k)
= \langle \nabla E(\bs z_{k+1}),\,\bs z_{k+1}-\bs z_k\rangle
- D_E(\bs z_k,\bs z_{k+1}),
\end{equation*}
where
\begin{equation}\label{eq:DEHNAG}
- D_E(\bs z_k,\bs z_{k+1})
= -D_f(x_k,x_{k+1})-\frac{\mu}{2}\|y_k-y_{k+1}\|^2 .
\end{equation}

Write the deviation from the implicit Euler step as
$$
\bs z_{k+1}-\bs z_k
= \alpha \mathcal G(\bs z_{k+1})
+ \alpha
\begin{pmatrix}
y_k-y_{k+1}+\beta(\nabla f(x_{k+1})-\nabla f(x_k))\\
0
\end{pmatrix}.
$$
\medskip
\noindent
\textbf{The implicit Euler term.} We apply the strong Lyapunov property \eqref{eq:A-HNAG} at $\bs z_{k+1}$ to the implicit Euler term
\begin{equation*}
\langle \nabla E(\bs z_{k+1}),\,\alpha \mathcal G(\bs z_{k+1}) \rangle \leq - \alpha E(\boldsymbol z_{k+1}) - \alpha\beta\|\nabla f(x_{k+1})\|^2-\frac{\alpha \mu}{2}\|x_{k+1} - y_{k+1}\|^2.
\end{equation*}

\medskip
\noindent
\textbf{Lagging of \(y\).}
%
For the term $\alpha\langle \nabla f(x_{k+1}),\,y_k-y_{k+1}\rangle$ arising from the lagging of $y$, the Cauchy--Schwarz and Young inequalities give
\begin{equation}\label{eq:firstcross}
\alpha\langle \nabla f(x_{k+1}),\,y_k-y_{k+1}\rangle
\le \frac{\alpha^2}{2\mu}\|\nabla f(x_{k+1})\|^2
+ \frac{\mu}{2}\|y_k-y_{k+1}\|^2,
\end{equation}
whose second term cancels with \eqref{eq:DEHNAG}. 

\medskip
\noindent
\textbf{Lagging of \(\nabla f(x)\).}
For the gradient cross term, arising from the lagging of $\nabla f(x)$,
the identity of squares yields
$$
\begin{aligned}
\alpha\beta\langle \nabla f(x_{k+1}),\,\nabla f(x_{k+1})-\nabla f(x_k)\rangle
&= -\frac{\alpha\beta}{2}\|\nabla f(x_k)\|^2
+ \frac{\alpha\beta}{2}\|\nabla f(x_{k+1})\|^2 \\
&\quad + \frac{\alpha\beta}{2}\|\nabla f(x_{k+1})-\nabla f(x_k)\|^2 .
\end{aligned}
$$
Combining and rearranging terms yields \eqref{eq:onestep}.
$\Box$ \end{proof}

We now select parameters to obtain linear convergence. 

\begin{theorem}[$(1-1/\sqrt{\kappa})$-linear convergence]\label{thm:linearconvHNAG-funcval-first}
Let $f\in\mathcal S_{\mu,L}$. For the iterates $(x_k, y_k)$ generated by~\eqref{eq:HNAG} with
$\alpha=\sqrt{\mu/L}$ and $\alpha\beta=1/L$,
and $E$ defined by \eqref{eq:lyapunov_strong_HNAG}, we have
\begin{equation}\label{eq:HNAGrate-first}
E(x_{k+1},y_{k+1})
\le \frac{1}{1+\sqrt{1/\kappa}}\,E(x_k,y_k),\quad k\ge 0.
\end{equation}
\end{theorem}
\begin{proof}
By co-coercivity \eqref{eq:cocov}, the choice $\alpha\beta=1/L$ implies
$$
\frac{\alpha\beta}{2}\|\nabla f(x_{k+1})-\nabla f(x_k)\|^2 - D_f(x_k,x_{k+1}) \le 0.
$$
With $\alpha^2=\mu/L$, the coefficient $\frac{\alpha^2}{2\mu}-\frac{\alpha\beta}{2}$ in front of $\|\nabla f(x_{k+1})\|^2$ vanishes.
Discarding all nonpositive terms in \eqref{eq:onestep} yields \eqref{eq:HNAGrate-first}. 
$\Box$ \end{proof}

\RV{Lemma~\ref{lem:disc_strong} is a leading-order estimate. The proof leaves two favorable terms unused: the negative term $-\frac{\alpha\mu}{2}\|x_{k+1}-y_{k+1}\|^2$, and an $O(\alpha^2)$ gain lost in the ``Lagging of $y$'' step by using inequality \eqref{eq:firstcross} instead of the exact $y$-update identity. Retaining them improves only higher-order constants in the step size and contraction factor; the leading-order rate is unchanged. This improvement is not visible numerically for ill-conditioned problems.}

More importantly, the rate can be improved by increasing the coercivity constant $c_L$ through the primal shift, or by enlarging the admissible step size $\alpha$ through the dual shift.

\section{A Family of HNAG-Type Flows and Schemes}
In this section, we present a family of HNAG-type flows and their discretizations. We show that the continuous-time coercivity can be improved by optimizing the time-rescaling parameters.

\subsection{Optimizing the coercivity of the HNAG-type flow}\label{sec:HNAGparametricflow}

Consider time-rescaling parameters $\tau_x$ and $\tau_y$ in the $x$- and $y$-updates:
$$
\begin{cases}
x' = \tau_x (y-x)-\RV{\beta\nabla f(x)},\\[1mm]
y' = \tau_y (x-y)-\dfrac1\mu \nabla f(x).
\end{cases}
$$
\RV{We define the corresponding vector field
$$
\mathcal G_p(\boldsymbol z)
:=
\left(
\tau_x(y-x)-\beta\nabla f(x),\,
\tau_y(x-y)-\frac1\mu\nabla f(x)
\right)^{\intercal}.
$$}
We pair this flow with the parametrized Lyapunov candidate
$$
\mathcal E_{p}(x,y)
=
D_{f_{-\theta\mu}}(x,x^\star)
+
\frac{\eta\mu}{2}\|y-x^\star\|^2,
$$
where the parameters $p=(\tau_x,\tau_y,\theta,\eta)$ are to be chosen.

\RV{We use
$$
\nabla f(x)=\nabla f_{-\theta\mu}(x)+\theta\mu(x-x^\star).
$$}
Following Lemma~\ref{lem:cont_strong}, the $x$-component contribution to $-\langle \nabla \mathcal E_p(\boldsymbol z),\mathcal G_p(\boldsymbol z)\rangle$ is
$$
\begin{aligned}
&-\left\langle \nabla f_{-\theta\mu}(x),\,\tau_x(y-x)-\beta\nabla f(x)\right\rangle
\\
={}&
-\tau_x\left\langle \nabla f_{-\theta\mu}(x),\,y-x^\star\right\rangle
+
(\tau_x+\beta\mu\theta)\left\langle \nabla f_{-\theta\mu}(x),\,x-x^\star\right\rangle
+
\beta\|\nabla f_{-\theta\mu}(x)\|^2 .
\end{aligned}
$$
The $y$-component contribution is
$$
\begin{aligned}
&-\eta\mu\left\langle y-x^\star,\,
\tau_y(x-y)-\tfrac1\mu\nabla f(x)
\right\rangle
\\
={}&
-\eta\mu\left\langle y-x^\star,\,
-\tau_y(y-x^\star)
+
(\tau_y-\theta)(x-x^\star)
-
\tfrac1\mu\nabla f_{-\theta\mu}(x)
\right\rangle
\\
={}&
\eta\mu\tau_y\|y-x^\star\|^2
-
\eta\mu(\tau_y-\theta)\langle y-x^\star,x-x^\star\rangle
+
\eta\left\langle y-x^\star,\nabla f_{-\theta\mu}(x)\right\rangle,
\end{aligned}
$$
with 
$$
2\langle y-x^\star,x-x^\star\rangle
=
\|y-x^\star\|^2
+
\|x-x^\star\|^2
-
\|x-y\|^2.
$$
We impose $\eta=\tau_x$ to cancel the mixed gradient term
$$
(\eta-\tau_x)\langle \nabla f_{-\theta\mu}(x),\,y-x^\star\rangle = 0.
$$
With this condition imposed, we have
\begin{equation}\label{eq:opt-taux-tauy-cancelled}
\begin{aligned}
&-\langle \nabla \mathcal E_{p}(\boldsymbol z),\mathcal G_p(\boldsymbol z)\rangle
={}
(\eta+\beta\mu\theta)\langle \nabla f_{-\theta\mu}(x),x-x^\star\rangle
+
\beta\|\nabla f_{-\theta\mu}(x)\|^2
\\
&\quad
+
\frac{\eta\mu}{2}
\Bigl(
(\tau_y+\theta)\|y-x^\star\|^2
-
(\tau_y-\theta)\|x-x^\star\|^2
+
(\tau_y-\theta)\|x-y\|^2
\Bigr).
\end{aligned}
\end{equation}
To obtain the strong Lyapunov property, we use two ways to control the negative square terms.

\paragraph{Using convexity.}
\RV{To drop the term $\|x-y\|^2$, we impose $\theta\le \tau_y$.} Since $f_{-\theta\mu}$ is $(1-\theta)\mu$-strongly convex for $\theta\leq 1$, we have
\begin{equation*}
\eta D_{f_{-\theta\mu}}(x^\star,x)
-
\frac{\eta(\tau_y-\theta)\mu}{2}\|x-x^\star\|^2
\ge
\frac{\eta(1-\tau_y)\mu}{2}\|x-x^\star\|^2 .
\end{equation*}
Since $f$ is $\mu$-strongly convex, the largest admissible shift is $\mu$, that is, $\theta\le 1$. To safely drop the quadratic terms involving $\|x-x^\star\|^2$, we assume $0\leq \tau_y\leq 1$. Thus
$$
0\leq \theta\le \tau_y\le 1.
$$
Under these conditions, we may take
$$
c_L=\min\{\eta+\beta\mu\theta,\tau_y+\theta\}\le 2.
$$
The largest possible value $c_L=2$ is obtained by taking
$$
\tau_x=\eta=2,\quad \tau_y=\theta=1.
$$
We use this setting in Section~\ref{sec:HNAG+}, which leads to HNAG$^+$.

\paragraph{Using symmetry.}
We write
$$
D_{f_{-\theta\mu}}(x^\star,x)
=
D_{f_{-\theta\mu}}(x,x^\star)+\Delta_f(x,x^\star),
$$
where the quadratic shift does not change the Bregman asymmetry. Imposing $\tau_y=\theta$ in \eqref{eq:opt-taux-tauy-cancelled} gives
\begin{equation*}
\begin{aligned}
-\langle \nabla \mathcal E_{p}(\boldsymbol z),\mathcal G_p(\boldsymbol z)\rangle
={}&
2(\eta+\beta\mu\theta)D_{f_{-\theta\mu}}(x,x^\star)
+
2\theta\frac{\eta\mu}{2}\|y-x^\star\|^2
\\
&\quad
+
(\eta+\beta\mu\theta)\Delta_f(x,x^\star)
+
\beta\|\nabla f_{-\theta\mu}(x)\|^2 .
\end{aligned}
\end{equation*}
\RV{Ignoring the sign-indefinite remainder $(\eta+\beta\mu\theta)\Delta_f(x,x^\star)$, which is treated rigorously in Section~\ref{sec:HNAG++} using the LAS assumption}, we may take
$$
c_L=\min\{2(\eta+\beta\mu\theta),2\theta\}\le 2.
$$
The largest possible value $c_L=2$ is obtained by taking
$$
\tau_x=\eta=1,\quad \tau_y=\theta=1.
$$
The remaining term $\Delta_f(x,x^\star)$ requires refined analysis. We use this setting in Section~\ref{sec:HNAG++}, which leads to HNAG$^{++}$.

\RV{The positive parameter $\beta$ improves the $x$-dissipation and helps control explicit-gradient lagging terms after discretization. It does not increase $c_L$, because the $y$-component remains the bottleneck. Moreover, due to the lagging of $\nabla f$, the discrete scheme requires $\alpha\beta$ to be controlled by smoothness; taking a large $\beta$ therefore forces a smaller step size $\alpha$.}

The strong Lyapunov property \eqref{eq:A-HNAG} of HNAG corresponds to
$$
\tau_x=\tau_y=\eta=1,\quad \theta=0,\quad c_L=1,
$$
which is not optimized for $c_L$.
 
\subsection{Discretization with parameters}
Following the above discussion, we fix \(\tau_y=\theta=1\), leave \(\tau_x=\tau\) as a parameter, and consider the family of discretizations
\begin{subequations}\label{eq:scheme1_scaled}
\begin{align}
\frac{x_{k+1}-x_k}{\alpha}
&=
\tau(y_k-x_{k+1})-\beta\nabla f(x_k),
\\
\frac{y_{k+1}-y_k}{\bar\alpha}
&=
x_{k+1}-y_{k+1}
-\frac1\mu\nabla f(x_{k+1}),
\end{align}
\end{subequations}
with positive parameters \((\tau,\bar\alpha,\alpha,\beta)\). 

Introduce one gradient descent step
$$
x_k^+:=x_k-\frac1L\nabla f(x_k).
$$
Eliminating \(y_k\) from \eqref{eq:scheme1_scaled} gives the two-step iteration
\begin{equation}\label{eq:twostep}
x_{k+1}=x_k+c_1(x_k^+-x_k)+c_2(x_k^+-x_{k-1}^+)+c_3(x_k-x_{k-1}),
\end{equation}
where
$$
c_1=\frac{\bar\alpha\alpha L\left(\beta+\frac{\tau}{\mu}\right)}{(1+\alpha\tau)(1+\bar\alpha)},\qquad
c_2=\frac{\alpha\beta L}{(1+\alpha\tau)(1+\bar\alpha)},\qquad
c_3=\frac{1-\alpha\beta L}{(1+\alpha\tau)(1+\bar\alpha)}.
$$
Under the scaling
$$
\bar \alpha \sim \alpha \sim \sqrt{\rho}, \qquad \alpha \beta = \frac{1}{L}, \qquad \alpha \beta \mu \sim \rho, \qquad \beta \mu \sim \sqrt{\rho} \sim \alpha.
$$
%
we have \(c_3=0\), \(c_1=O(1)\), and
$$
c_2=\frac1{(1+\alpha\tau)(1+\bar\alpha)}
=1-\alpha\tau-\bar\alpha+O(\rho)
=1-(1+\tau)\sqrt{\rho}+O(\rho).
$$

\begin{remark}
The term \(c_2(x_k^+-x_{k-1}^+)\) contains
$$
-\frac{c_2}{L}\bigl(\nabla f(x_k)-\nabla f(x_{k-1})\bigr)
\approx
-\frac{c_2}{L}\nabla^2 f(\xi)(x_k-x_{k-1}),
$$
which captures the change in the gradient. The Hessian form of this correction motivates the name HNAG.
\end{remark}


Conversely, some classical two-step methods, including NAG~\cite{nesterov2003introductory}, can be written in the HNAG-type form \eqref{eq:scheme1_scaled}. Indeed, NAG corresponds to
$$
x_k^+=x_k-\frac1L\nabla f(x_k),
\qquad
x_{k+1}=x_k^+
+\frac{\sqrt{\kappa}-1}{\sqrt{\kappa}+1}
\left(x_k^+-x_{k-1}^+\right),
$$
which is 
equivalent to \(\tau=1\), \(\alpha\beta=1/L\), $\alpha = \sqrt{\mu/L}$, and
\(\bar\alpha=\alpha/(1-\alpha)\). Thus, NAG is HNAG with a slightly larger step size in the $y$-update.


We list several accelerated methods discussed in this paper in Table~\ref{tab:hnag_equivalent_forms}.
\begin{table}[htbp]
\centering
\renewcommand{\arraystretch}{2.1}
\setlength{\tabcolsep}{10pt}
\begin{tabular}{@{}l c c@{}}
  \toprule
  Method & $(c_1,c_2,c_3)$ & $(\tau,\bar\alpha,\alpha,\alpha\beta)$ \\
  \midrule

  HNAG
  & $\left(\dfrac1{1+a},\dfrac1{(1+a)^2},0\right)$
  & $\left(1,a,a,\dfrac1L\right)$ \\[2mm]

  NAG
  & $\left(1,\dfrac{1-a}{1+a},0\right)$
  & $\left(1,\dfrac{a}{1-a},a,\dfrac1L\right)$ \\[2mm]

  TM / HNAG$^+$
  & $\left(\dfrac{2+a-a^2}{1+a},\dfrac{(1-a)^2}{1+a},0\right)$
  & $\left(2,\dfrac{a}{1-a},\dfrac{a}{1-a},\dfrac1L\right)$ \\[2mm]

  HNAG$^{++}$
  & $\left(\dfrac{2+\sqrt2\,a}{(1+\sqrt2\,a)^2},\dfrac1{(1+\sqrt2\,a)^2},0\right)$
  & $\left(1,\sqrt2\,a,\sqrt2\,a,\dfrac1L\right)$ \\[2mm]

  \bottomrule
\end{tabular}
\caption{Equivalent two-step \eqref{eq:twostep} and HNAG-type formulations \eqref{eq:scheme1_scaled} of accelerated methods, where \(a=\sqrt{\rho}=1/\sqrt{\kappa}\).}
\label{tab:hnag_equivalent_forms}
\end{table}

Rather than analyze the full parametric family \eqref{eq:scheme1_scaled}, which would be cumbersome because the parameters are underdetermined, we focus on two representative cases: $\tau=2$ in Section~\ref{sec:HNAG+} and $\tau=1$ in Section~\ref{sec:HNAG++}. We also assume $\bar\alpha=\alpha$ for simplicity and outline the possible generalization to $\bar\alpha\ne\alpha$.

\section{HNAG$^+$: Method and Convergence}\label{sec:HNAG+}

By the discussion in Section~\ref{sec:HNAGparametricflow}, we consider the rescaled HNAG flow
\begin{equation}\label{eq:flow_scaled}
\left\{
\begin{aligned}
x' &= 2(y-x)-\beta\,\nabla f(x),\\
y' &= x-y-\frac{1}{\mu}\nabla f(x),
\end{aligned}
\right.
\end{equation}
Define
$$
\mathcal G^+(\boldsymbol z)
:=
\left(
2(y-x)-\beta\nabla f(x),\,
x-y-\frac1\mu\nabla f(x)
\right)^{\intercal}.
$$
The factor \(2\) in the \(x\)-dynamics reflects the optimal coercivity.

\subsection{Discretization and Algorithm}
Discretizing~\eqref{eq:flow_scaled} using the same implicit--explicit splitting
as in HNAG gives
\begin{subequations}\label{eq:scheme2_scaled}
\begin{align}
\frac{x_{k+1} - x_k}{\alpha} &= 2(y_k - x_{k+1}) - \beta \nabla f(x_k), \label{eq:scheme2_scaled_x} \\
\frac{y_{k+1} - y_k}{\alpha} &= x_{k+1} - y_{k+1} - \frac{1}{\mu} \nabla f(x_{k+1}).
\end{align}
\end{subequations}

Similar to HNAG, the HNAG$^+$ scheme admits an implementation-friendly form in which the
gradient is evaluated only once per iteration.  
We still use $\alpha\beta=1/L$ but the step size $\alpha =\sqrt{\rho}/(1-\sqrt{\rho})$ is slightly larger.

%

\vskip -10pt
\begin{algorithm}
\caption{HNAG$^+$}
\label{alg:HNAGplus}
\linespread{1.3}\selectfont
\begin{algorithmic}[1]
\Require $x_0,v_0\in\mathbb R^d$, parameters $L>\mu>0$
\State $\alpha\gets\sqrt{\mu/L}/(1-\sqrt{\mu/L})$
\For{$k=0,1,2,\ldots$}
\State $v_{k+1}\gets \tfrac{1}{1+\alpha}\!\left(v_k+\alpha^2 x_k-\tfrac{1}{(\sqrt{L}-\sqrt{\mu})^2}\nabla f(x_k)\right)$
\State $x_{k+1}\gets \tfrac{1}{1+2\alpha}\!\left(x_k+2v_{k+1}-\tfrac{1}{L}\nabla f(x_k)\right)$
\EndFor
\State \textbf{Return:} $x_{k+1}$ and $y_{k} = v_{k+1}/\alpha$
\end{algorithmic}
\end{algorithm}


\subsection{Strong Lyapunov Property}
We define the Lyapunov energy
\begin{equation}\label{eq:lyaderivation_scaled}
\mathcal{E}^+(\boldsymbol z)=\mathcal{E}(x,y)
:= D_{f_{-\mu}}(x,x^{\star}) + \mu \|y - x^{\star}\|^2 .
\end{equation}
Compared with~\eqref{eq:lyapunov_strong_HNAG}, this definition incorporates a primal shift $f\mapsto f_{-\mu}$ and doubles the weight of the $y$--term. \RV{Strictly speaking, $\mathcal E^+$ need not be positive definite in $\boldsymbol z$: $\mathcal E^+(\boldsymbol z)=0$ implies $y=x^\star$, but may not imply $x=x^\star$, because the shifted function $f_{-\mu}$ is convex but not necessarily strongly convex. We therefore call $\mathcal E^+$ a Lyapunov energy rather than a Lyapunov function.}
We refine the strong Lyapunov property by incorporating the next-order terms.

\begin{lemma}\label{lem:energy_inequality_scaled}
For the Lyapunov energy $\mathcal E^+$ defined in \eqref{eq:lyaderivation_scaled} and the vector field $\mathcal G^+$ associated with the flow \eqref{eq:flow_scaled}, the following strong Lyapunov property holds:
\begin{equation}\label{eq:energy_identity_scaled}
   \begin{aligned}
&-\big\langle \nabla \mathcal E^+(\boldsymbol z),\,\mathcal G^+(\boldsymbol z)\big\rangle\\
={} &
(2+\beta\mu)\,\langle \nabla f_{-\mu}(x),\,x-x^{\star}\rangle
+2\mu\|y-x^{\star}\|^2
+\beta\|\nabla f_{-\mu}(x)\|^2\\
={} & 2\,\mathcal E^+(\bs z) +\beta\|\nabla f_{-\mu}(x)\|^2 + 
(2+\beta\mu)D_{f_{-\mu}} (x^{\star}, x) + \beta\mu D_{f_{-\mu}} (x, x^{\star}).
   \end{aligned}
\end{equation}
\end{lemma}

\begin{proof}
A direct calculation gives
$$
\begin{aligned}
-\big\langle\nabla \mathcal E^+(\boldsymbol z),\mathcal G^+(\boldsymbol z)\big\rangle
&= - \left\langle
\begin{pmatrix}
\nabla f_{-\mu}(x)\\[2pt]
2\mu (y-x^{\star})
\end{pmatrix}
,
\begin{pmatrix}
2(y-x)-\beta\nabla f(x)\\[2pt]
x-y-\tfrac{1}{\mu}\nabla f(x)
\end{pmatrix} \right\rangle\\[6pt]
&=\left\langle
\begin{pmatrix}
\nabla f_{-\mu}(x)\\[2pt]
2\mu (y-x^{\star})
\end{pmatrix}
,
\begin{pmatrix}
2(x-y)
+\beta\nabla f_{-\mu}(x)+\beta\mu(x-x^{\star})\\[2pt]
(y-x^{\star})+\tfrac{1}{\mu}\nabla f_{-\mu}(x)
\end{pmatrix} \right\rangle \\[6pt]
&=
(2+\beta\mu)\,\langle \nabla f_{-\mu}(x),\,x-x^{\star}\rangle
+2\mu\|y-x^{\star}\|^2
+\beta\|\nabla f_{-\mu}(x)\|^2 .
\end{aligned}
$$
Since $\nabla f_{-\mu}(x^{\star})=0$, the symmetrized Bregman identity yields
$$
\begin{aligned}
\langle \nabla f_{-\mu}(x),\,x-x^{\star}\rangle
&{}=
D_{f_{-\mu}}(x,x^{\star})+D_{f_{-\mu}}(x^{\star},x),
\end{aligned}
$$
Combining the terms gives \eqref{eq:energy_identity_scaled}.
$\Box$ \end{proof}

Lemma~\ref{lem:energy_inequality_scaled} shows that the scaled HNAG$^+$ flow satisfies a strong Lyapunov property with coercivity constant \(c_L=2\) \RV{as predicted in Section~\ref{sec:HNAGparametricflow}}. 

\subsection{Convergence}

We establish the linear convergence of the HNAG$^+$ scheme. The proof follows the same framework as in Lemma~\ref{lem:disc_strong}, with changes due to the larger coercivity constant \(c_L=2\) and the shifted Lyapunov functional. To allow a larger step size, we also retain the next-order terms in the estimate.

\begin{lemma}\label{lm:decay1-HNAGplus}
Let $\boldsymbol z_k=(x_k,y_k)$ be the iterates generated by the
HNAG$^+$~\eqref{eq:scheme2_scaled}.
Then the following one-step inequality holds:
\begin{equation}\label{eq:HNAG+onestep}
\begin{aligned}
&\mathcal{E}^+(\boldsymbol z_{k+1}) - \mathcal{E}^+(\boldsymbol z_k)\\
\;\le\;& - 2\alpha D_{f_{-\mu}}(x_{k+1},x^{\star}) - (2\alpha+\alpha^2)\mu\|y_{k+1}-x^{\star}\|^2 \\
&{}+\Bigl(\frac{\alpha^2}{\mu}
-\frac{\alpha\beta}{2} - \frac{\alpha(2+\beta\mu)}{2(L-\mu)} \Bigr)
\|\nabla f_{-\mu}(x_{k+1})\|^2
-\frac{\alpha\beta}{2}\frac{L}{L-\mu} \|\nabla f_{-\mu}(x_k)\|^2 \\[4pt]
&
+\frac{\alpha\beta}{2}
\|\nabla f_{-\mu}(x_{k+1})-\nabla f_{-\mu}(x_k)\|^2
-(1-\alpha\beta\mu)D_{f_{-\mu}}(x_k,x_{k+1}).
\end{aligned}
\end{equation}
\end{lemma}
\begin{proof}
Using the Bregman expansion,
\begin{equation}\label{eq:DE_HNAGplus}
\mathcal E^+(\bs z_{k+1})-\mathcal E^+(\bs z_k)
=
\big\langle \nabla \mathcal E^+(\bs z_{k+1}),\,\bs z_{k+1}-\bs z_k\big\rangle
-
D_{\mathcal E^+}(\bs z_k,\bs z_{k+1}).
\end{equation}
Let $g_k:=\nabla f_{-\mu}(x_k)$.
Write the update as a correction to implicit Euler scheme
$$
\bs z_{k+1}-\bs z_k
=
\alpha \mathcal G^+(\bs z_{k+1})
+
\alpha
\binom{
2(y_k-y_{k+1})
+\beta(g_{k+1}-g_k)
+\beta\mu(x_{k+1}-x_k)
}{0}.
$$
\medskip
\noindent
\textbf{The implicit Euler term.} The implicit Euler part is controlled by Lemma~\ref{lem:energy_inequality_scaled}:
\begin{equation}\label{eq:IE-HNAGplus}
\begin{aligned}
\alpha\big\langle \nabla \mathcal E^+(\bs z_{k+1}),\,\mathcal G^+(\bs z_{k+1})\big\rangle
={}&
-2\alpha\,\mathcal E^+(\bs z_{k+1}) - \alpha\beta \|g_{k+1}\|^2\\
&
-\alpha\beta\mu D_{f_{-\mu}}(x_{k+1},x^\star) - \alpha(2+\beta\mu)D_{f_{-\mu}} (x^{\star}, x_{k+1}).
\end{aligned}
\end{equation}

It remains to estimate the three correction terms.

\medskip
\noindent
\textbf{The \(\mu\)-shift term.}
By the three-point Bregman identity~\eqref{eq:Bregmanidentity},
$$
\begin{aligned}
\alpha\beta\mu\langle g_{k+1},x_{k+1}-x_k\rangle
={}&
\alpha\beta\mu\Bigl(
D_{f_{-\mu}}(x_{k+1},x^\star)
+
D_{f_{-\mu}}(x_k,x_{k+1})
-
D_{f_{-\mu}}(x_k,x^\star)
\Bigr).
\end{aligned}
$$
The first term cancels with $-\alpha\beta\mu D_{f_{-\mu}}(x_{k+1},x^\star)$ in \eqref{eq:IE-HNAGplus}. The second term is absorbed into
$-D_{\mathcal E^+}(\bs z_k,\bs z_{k+1})$
to give
$$
-(1-\alpha\beta\mu)D_{f_{-\mu}}(x_k,x_{k+1}).
$$
The last term is nonpositive and may be relaxed using
\begin{equation}\label{eq:extragk}
-\alpha\beta\mu D_{f_{-\mu}}(x_k,x^\star)
\le
-\frac{\alpha\beta\mu}{2(L-\mu)}\|g_k\|^2,
\end{equation}
by the \((L-\mu)\)-smoothness of \(f_{-\mu}\).

\medskip
\noindent
\textbf{Lagging of \(y\).}
Using the \(y\)-update,
$$
\frac{y_k-y_{k+1}}{\alpha}-\frac1\mu g_{k+1}=y_{k+1}-x^\star,
$$
we square and rescale both sides to get
$$
2\alpha\langle g_{k+1},y_k-y_{k+1}\rangle
=
\mu\|y_k-y_{k+1}\|^2
+\frac{\alpha^2}{\mu}\|g_{k+1}\|^2
-\alpha^2\mu\|y_{k+1}-x^\star\|^2 .
$$
The term $\mu\|y_k-y_{k+1}\|^2$ is canceled by $-D_{\mathcal E^+}(\bs z_k,\bs z_{k+1})$.
The last negative term is merged into $\|y_{k+1}-x^\star\|^2$ in \eqref{eq:IE-HNAGplus}.

\medskip
\noindent
\textbf{Lagging of the gradient.}
Using the identity of squares,
$$
\begin{aligned}
\alpha\beta\langle g_{k+1},g_{k+1}-g_{k}\rangle
={}&
-\frac{\alpha\beta}{2}\|g_k\|^2
+\frac{\alpha\beta}{2}\|g_{k+1}\|^2
+\frac{\alpha\beta}{2}\|g_{k+1}-g_k\|^2.
\end{aligned}
$$
The negative term $-\frac{\alpha\beta}{2}\|g_k\|^2$ is merged with the right hand side of \eqref{eq:extragk} to update the coefficient $- \frac{\alpha\beta}{2} - \frac{\alpha\beta\mu}{2(L-\mu)} = - \frac{\alpha\beta}{2} \frac{L}{L-\mu}$ in front of $\|g_k\|^2$.

The positive term $
\frac{\alpha\beta}{2}\|g_{k+1}\|^2$
cancels half of $-\alpha\beta\|g_{k+1}\|^2$ in \eqref{eq:IE-HNAGplus}. The negative term $- \alpha(2+\beta\mu)D_{f_{-\mu}} (x^{\star}, x_{k+1})$ in \eqref{eq:IE-HNAGplus} will contribute more negative $\|g_{k+1}\|^2$ by co-coercivity and leads the coefficient of $\|g_{k+1}\|^2$ in \eqref{eq:HNAG+onestep}. 

\medskip

Combining \eqref{eq:DE_HNAGplus}, \eqref{eq:IE-HNAGplus}, and the three estimates above yields the desired one-step estimate.
$\Box$\end{proof}

We now select parameters to obtain linear convergence.

\begin{theorem}[$(1-2/\sqrt{\kappa})$-linear convergence]\label{thm:quadconv-HNAGplus}
Let $f\in\mathcal S_{\mu,L}$ with $L>\mu>0$.
For the iterates $(x_k, y_k)$ generated by HNAG$^+$~\eqref{eq:scheme2_scaled} with parameters
$$
\alpha\beta=\frac{1}{L},
\qquad
\alpha= \frac{1/\sqrt{\kappa}}{1-1/\sqrt{\kappa}} = \frac{1}{\sqrt{\kappa}-1},
$$
and the Lyapunov energy $\mathcal E^+$ defined by \eqref{eq:lyaderivation_scaled}, we have
$$
\begin{aligned}
&\mathcal{E}^+(x_{k+1},y_{k+1})
-\frac{1}{2(L-\mu)}\|\nabla f_{-\mu}(x_{k+1})\|^2\\
\le{}&
\frac{\sqrt{\kappa}-1}{\sqrt{\kappa}+1} 
\left(
\mathcal{E}^+(x_k,y_k)
-\frac{1}{2(L-\mu)}\|\nabla f_{-\mu}(x_k)\|^2
\right).
\end{aligned}
$$
\end{theorem}
\begin{proof}
Since $f_{-\mu}$ is $(L-\mu)$-smooth, when $\alpha\beta=1/L$, we have
$$
\frac{\alpha\beta}{2}\|\nabla f_{-\mu}(x_{k+1})-\nabla f_{-\mu}(x_k)\|^2
\le
(1-\alpha\beta\mu)D_{f_{-\mu}}(x_k,x_{k+1}).
$$
By Lemma~\ref{lm:decay1-HNAGplus}, we have
\begin{equation*}
\begin{aligned}
(1+2\alpha)\mathcal{E}^+(\boldsymbol z_{k+1})  
\le{}\;&
\mathcal{E}^+(\boldsymbol z_k)
-\frac{1}{2(L-\mu)}
\|\nabla f_{-\mu}(x_k)\|^2
\\
&\quad
+\left(
\frac{\alpha^2}{\mu}
-\frac{\alpha\beta}{2}
-\frac{\alpha(2+\beta\mu)}{2(L-\mu)}
\right)
\|\nabla f_{-\mu}(x_{k+1})\|^2.
\end{aligned}
\end{equation*}
Define the dual-shifted energy
$$
\widetilde{\mathcal E}^+_k
:=
\mathcal E^+(x_k,y_k)-\frac{1}{2(L-\mu)}\|\nabla f_{-\mu}(x_k)\|^2.
$$
\RV{Since $f_{-\mu}$ is convex and $(L-\mu)$-smooth,
$\widetilde{\mathcal E}^+_k
\geq
\mu\|y_k-x^\star\|^2
\geq0.$}
A sufficient condition for
$$
(1+2\alpha)\widetilde{\mathcal E}^+_{k+1}\le \widetilde{\mathcal E}^+_k
$$
is that the coefficient of $\|\nabla f_{-\mu}(x_{k+1})\|^2$ satisfies

\begin{equation}\label{eq:key}
\frac{\alpha^2}{\mu}
-\frac{\alpha\beta}{2}
-\frac{\alpha(2+\beta\mu)}{2(L-\mu)}
\le
\frac{1+2\alpha}{2(L-\mu)}.
\end{equation}

Solving \eqref{eq:key} with $\alpha\beta=1/L$ shows that the largest admissible step size is
$$
\alpha=\frac{1/\sqrt{\kappa}}{1-1/\sqrt{\kappa}}=\frac{1}{\sqrt{\kappa}-1}.
$$
Substituting into the rate $1/(1+2\alpha)$ yields the rate $(\sqrt{\kappa}-1)/(\sqrt{\kappa}+1)$.
$\Box$\end{proof}

\begin{remark}
After rewriting the method in the two-step form \eqref{eq:twostep}, one can show that the Triple Momentum method~\cite{van2017fastest} is equivalent to HNAG$^+$, Algorithm~\ref{alg:HNAGplus}, for a particular choice of parameters. The design principle and the proof here are, however, different from the original TM analysis.

We next show how a modified discretization can further increase the coercivity in the shifted Lyapunov estimate. \RV{The $y$-update is kept unchanged}. Consider \RV{a modified} HNAG$^+$ $x$-update
\begin{equation}\label{eq:scheme2_scaled_perturbed}
\frac{x_{k+1}-x_k}{\alpha}
=
2(y_k-x_{k+1})-\beta\nabla f(x_k)-\beta\mu(x_{k+1}-x_k).
\end{equation}
Moving the perturbation term in \eqref{eq:scheme2_scaled_perturbed} to the left-hand side yields the standard HNAG$^+$ $x$-update \eqref{eq:scheme2_scaled_x} with the reduced effective step size $\tilde\alpha = \alpha/(1+\alpha\beta\mu) < \alpha$.
Let $g_k:=\nabla f_{-\mu}(x_k)$. Then
$$
-\beta\nabla f(x_k)-\beta\mu(x_{k+1}-x_k)
=
-\beta\nabla f(x_{k+1})+\beta(g_{k+1}-g_k).
$$
\RV{Consequently, the correction to the implicit Euler scheme becomes
$$
\bs z_{k+1}-\bs z_k
=
\alpha \mathcal G^+(\bs z_{k+1})
+
\alpha
\binom{
2(y_k-y_{k+1})
+\beta(g_{k+1}-g_k)
}{0}.
$$
Thus the perturbation removes the $\mu$-shift lagging term $\alpha\beta\mu\langle g_{k+1},x_{k+1}-x_k\rangle$ in the proof of Lemma~\ref{lm:decay1-HNAGplus}.}

\RV{In the corresponding refined estimate, the coercivity coefficient becomes
$$
c_L=\min\{2+\beta\mu,\,2+\alpha\}.
$$
Here $2+\beta\mu$ comes from the coefficient of $\langle \nabla f_{-\mu}(x),x-x^\star\rangle$ in \eqref{eq:energy_identity_scaled}, while $2+\alpha$ comes from the term $(2+\alpha)\mu\|y_{k+1}-x^\star\|^2$ in \eqref{eq:HNAG+onestep}. This yields a higher-order improvement in the admissible step size and contraction factor. The leading-order rate is, however, unchanged.}

This refinement reflects a design principle different from TM: the perturbation term $-\beta\mu(x_{k+1}-x_k)$ is induced by the $\mu$-shift in the Lyapunov function, and this shift strengthens the coercivity. Since this is only a higher-order improvement, we keep the simpler HNAG$^+$ scheme in Algorithm~\ref{alg:HNAGplus}.
\end{remark}

\begin{remark}
The estimate does not directly imply convergence of \(\|x_k-x^\star\|^2\), since the primal shift \(f\mapsto f_{-\mu}\) removes strong convexity in the \(x\)-component. In contrast, the sequence \(\{y_k\}\) satisfies
$$
\mu \|y_k-x^\star\|^2
\le
\widetilde{\mathcal E}^+_0
\left(\frac{\sqrt{\kappa}-1}{\sqrt{\kappa}+1}\right)^k,
$$
and therefore converges strongly to \(x^\star\). This is consistent with accelerated methods such as NAG, TM, and C2M, where the sequence with direct norm contraction is not the sequence at which the gradient is evaluated.
\end{remark}

\section{HNAG$^{++}$: A Larger-Step-Size Variant of HNAG}\label{sec:HNAG++}
In this section, we show that the larger step size $\alpha=\sqrt{2\mu/L}$ can be used. This step size is larger than the HNAG step size by a factor $\sqrt{2}\approx1.4$. \RV{Under the LAS condition, the resulting method, referred to as HNAG$^{++}$, attains the asymptotic rate $1-2\sqrt{2/\kappa}$.} 

\subsection{Algorithm}
HNAG$^{++}$ is defined as the original HNAG iteration~\eqref{eq:HNAG} executed with an enlarged step size. The algorithmic structure is unchanged; only the admissible step size (and hence the effective weights) is changed; see Algorithm~\ref{alg:HNAGpp}.

\begin{algorithm}
\caption{HNAG$^{++}$}
\label{alg:HNAGpp}
\linespread{1.3}\selectfont
\begin{algorithmic}[1]
\Require $x_0,v_0\in\mathbb R^d$, parameters $L\ge \mu>0$
\State $\alpha \gets \sqrt{2\mu/L}$
\For{$k=0,1,2,\ldots$}
\State $v_{k+1}\gets \tfrac{1}{1+\alpha}
\bigl(v_k+\alpha^2 x_k-\tfrac{2}{L}\nabla f(x_k)\bigr)$
\State $x_{k+1}\gets \tfrac{1}{1+\alpha}
\bigl(x_k+v_{k+1}-\tfrac{1}{L}\nabla f(x_k)\bigr)$
\EndFor
\State \textbf{Return:} $x_{k+1}$ and $y_{k} = v_{k+1}/\alpha$
\end{algorithmic}
\end{algorithm}


\subsection{Suboptimal linear rate}
We apply a shift in the dual space and obtain the following convergence result. A related result was obtained in~\cite[Theorem~8.1]{ChenLuo2021unified}, and an independent derivation was presented in~\cite{park2023factor_sqrt2_acceleration}.


\begin{theorem}[\((1-\sqrt{2/\kappa})\)-linear convergence]\label{thm:linearconvHNAG-funcval}
Let \(f\in\mathcal S_{\mu,L}\), and let \(E\) be defined by \eqref{eq:lyapunov_strong_HNAG}. For the iterates \((x_k,y_k)\) generated by the HNAG scheme~\eqref{eq:HNAG} with
$$
\alpha=\sqrt{2\mu/L},\quad \alpha\beta=1/L,
$$
or equivalently by Algorithm~\ref{alg:HNAGpp}, define
$$
\widetilde E_k:=E(x_k,y_k)-\frac1{2L}\|\nabla f(x_k)\|^2.
$$
Then, for all \(k\ge0\),
\begin{equation}\label{eq:HNAGrate}
\frac{\mu}{2}\|y_{k+1}-x^\star\|^2 \leq \widetilde E_{k+1}
\le
\frac1{1+\sqrt{2/\kappa}}\widetilde E_k
\le
\widetilde E_0\left(\frac1{1+\sqrt{2/\kappa}}\right)^{k+1}.
\end{equation}
Moreover, for all \(k\ge1\),
\begin{equation}\label{eq:HNAG-grad-bound}
\|\nabla f(x_k)\|^2
\le
\frac{2L(1+\alpha)}{\alpha}\,
\widetilde E_0
\left(\frac1{1+\alpha}\right)^k,
\end{equation}
and
\begin{equation}\label{eq:HNAG-strong-components}
\|x_k-x^\star\|^2
\le
\frac{2(1+2\alpha)}{\alpha\mu}\,
\widetilde E_0
\left(\frac1{1+\alpha}\right)^k .
\end{equation}
\end{theorem}

\begin{proof}
With \(\alpha\beta=1/L\) and \(\alpha^2=2\mu/L\), the one-step estimate \eqref{eq:onestep} becomes
$$
\begin{aligned}
(1+\alpha)E(\bs z_{k+1})
\le{}&
E(\bs z_k)
+\frac1{2L}\|\nabla f(x_{k+1})\|^2
-\frac1{2L}\|\nabla f(x_k)\|^2\\
&\  -\frac{\alpha\mu}{2}\|x_{k+1}-y_{k+1}\|^2 .
\end{aligned}
$$
Equivalently,
\begin{equation}\label{eq:one-step-energy-E}
(1+\alpha)\widetilde E_{k+1}
\le
\widetilde E_k
-\frac{\alpha\mu}{2}\|x_{k+1}-y_{k+1}\|^2
-\frac{\alpha}{2L}\|\nabla f(x_{k+1})\|^2 .
\end{equation}
Dropping the negative terms gives \eqref{eq:HNAGrate}.

Next, moving the negative terms in \eqref{eq:one-step-energy-E} to the left gives
$$
(1+\alpha)\widetilde E_{k+1}
+\frac{\alpha\mu}{2}\|x_{k+1}-y_{k+1}\|^2
+\frac{\alpha}{2L}\|\nabla f(x_{k+1})\|^2
\le
\widetilde E_k .
$$
Shifting the index, we get, for \(k\ge1\),
$$
\frac{\alpha}{2L}\|\nabla f(x_k)\|^2
\le
\widetilde E_{k-1}
\le
(1+\alpha)\widetilde E_0\left(\frac1{1+\alpha}\right)^k .
$$
This proves \eqref{eq:HNAG-grad-bound}. 
We use the $\mu$-strong convexity of $f$ to conclude 
$$
\frac{\mu}{2}\|x_k-x^{\star}\|^2 \leq \widetilde E_k + \frac1{2L}\|\nabla f(x_k)\|^2\leq
\frac{(1+2\alpha)}{\alpha}\,
\widetilde E_0
\left(\frac1{1+\alpha}\right)^k .
$$
This gives \eqref{eq:HNAG-strong-components}.
$\Box$\end{proof}


We next show that the coercivity can be improved by using the symmetry of Bregman divergence. 

\subsection{Quadratic convex functions}
To motivate, we start from quadratic convex functions. Assume that $f\in\mathcal S_{\mu,L}$ and that $f$ is quadratic. 
Consider the Lyapunov function
$$
\mathcal{E}(\boldsymbol z)=\mathcal{E}(x,y)
:= D_{f_{-\mu}}(x,x^{\star})+\frac{\mu}{2}\|y-x^{\star}\|^2.
$$
For quadratic functions, the Bregman divergence is symmetric. Consequently,
$$
\langle \nabla f_{-\mu}(x),\,x-x^{\star}\rangle
= D_{f_{-\mu}}(x,x^{\star})+D_{f_{-\mu}}(x^{\star},x)
= 2D_{f_{-\mu}}(x,x^{\star}).
$$
The associated HNAG flow \eqref{eq:flow} satisfies a strong Lyapunov property with coercivity constant $c_L=2$:
$$
-\big\langle \nabla \mathcal E(\boldsymbol z),\,\mathcal G(\boldsymbol z)\big\rangle
=
2\,\mathcal E(\boldsymbol z)
+\beta\|\nabla f_{-\mu}(x)\|^2
+\beta\mu\langle \nabla f_{-\mu}(x),\,x-x^{\star}\rangle .
$$
With this improved coercivity, the HNAG$^{++}$ iteration admits an accelerated linear convergence rate $(1+2\sqrt{2/\kappa})^{-1}$ for quadratic functions matching that of OGM-q~\cite{kim2018adaptive_restart_OGM}. This follows as a special case of the shifted Lyapunov estimate in Proposition~\ref{prop:full-shift}.

%

\subsection{Partial shift}
Since $f$ is $\mu$-strongly convex, for any $\hat{\mu}\leq \mu$, the shifted function
$$
f_{-\hat{\mu}}(x)
:= f(x) - \frac{\hat{\mu}}{2}\|x - x^{\star}\|^2
$$
is convex and satisfies $\nabla f_{-\hat{\mu}}(x^{\star}) = 0$. 
We define the Lyapunov function
\begin{equation}\label{eq:lya-muhat}
\mathcal{E}(\bs z; \hat{\mu})
:= D_{f_{-\hat{\mu}}}(x, x^{\star})
+ \tfrac{\mu}{2}\|y - x^{\star}\|^2.
\end{equation}
{Notice that the partial shift \(\hat\mu\) is used only in \(f\), while the \(y\)-component still uses the fixed coefficient \(\mu\).} 

We next state a refined strong Lyapunov property. {The key idea is to sacrifice part of the coercivity in order to control the asymmetry term \(\Delta_f(x,x^\star)\)}. 

\begin{lemma}\label{lem:energy_identity}
Let $0\leq\hat\mu\leq\mu$ and set
$\delta:=(\mu-\hat\mu)/\mu\in[0,1]$.
For the Lyapunov function $\mathcal E(\bs z;\hat\mu)$ defined in \eqref{eq:lya-muhat} and the vector field $\mathcal G$ defined in \eqref{eq:G}, we have
\begin{equation}\label{eq:energy_identity}
\begin{aligned}
-\big\langle \nabla \mathcal E(\bs z;\hat\mu),\mathcal G(\bs z)\big\rangle
={}&
\langle \nabla f_{-\hat\mu}(x),x-x^\star\rangle
+\mu\|y-x^\star\|^2
-\delta\mu\langle y-x^\star,x-x^\star\rangle
\\
&\quad
+\beta\|\nabla f_{-\hat\mu}(x)\|^2
+\beta\hat\mu\langle \nabla f_{-\hat\mu}(x),x-x^\star\rangle
\\
\geq{}&
\left(2-\sqrt{\delta}-b\right)\mathcal E(\bs z;\hat\mu)
+\RV{bD_{f_{-\hat\mu}}(x,x^\star)}+\beta\|\nabla f_{-\hat\mu}(x)\|^2\\
&\quad
+\beta\hat\mu\langle \nabla f_{-\hat\mu}(x),x-x^\star\rangle
+\left(1-\sqrt{\delta}\right)\Delta_f(x,x^\star),
\end{aligned}
\end{equation}
where $0\leq b<2-\sqrt{\delta}$.
\end{lemma}

\begin{proof}
A direct calculation gives
$$
\begin{aligned}
-\big\langle \nabla \mathcal E(\bs z;\hat\mu),\mathcal G(\bs z)\big\rangle
={}&
\left\langle
\begin{pmatrix}
\nabla f_{-\hat\mu}(x)\\[2pt]
\mu(y-x^\star)
\end{pmatrix},
\begin{pmatrix}
x-y+\beta\nabla f_{-\hat\mu}(x)+\beta\hat\mu(x-x^\star)\\
y-x^\star+\dfrac1\mu\nabla f_{-\hat\mu}(x)-\delta(x-x^\star)
\end{pmatrix}
\right\rangle
\\
={}&
\langle \nabla f_{-\hat\mu}(x),x-x^\star\rangle
+\mu\|y-x^\star\|^2
-\delta\mu\langle y-x^\star,x-x^\star\rangle
\\
&\quad
+\beta\|\nabla f_{-\hat\mu}(x)\|^2
+\beta\hat\mu\langle \nabla f_{-\hat\mu}(x),x-x^\star\rangle.
\end{aligned}
$$
This proves the identity. \RV{We keep the term $\beta\hat\mu\langle \nabla f_{-\hat\mu}(x),x-x^\star\rangle$ separate, since it controls the $\hat \mu$-shift term in Lemma~\ref{lm:decay1}.}

Since $f_{-\hat\mu}$ is $\delta\mu$-strongly convex,
\begin{equation}\label{eq:deltaconvexity}
D_{f_{-\hat\mu}}(x^\star,x)
\geq
\frac{\mu-\hat\mu}{2}\|x-x^\star\|^2
=
\frac{\delta\mu}{2}\|x-x^\star\|^2.
\end{equation}
Young's inequality and \eqref{eq:deltaconvexity} give
$$
\begin{aligned}
\delta\mu\langle y-x^\star,x-x^\star\rangle
&\leq
\sqrt{\delta}\,\frac{\mu}{2}\|y-x^\star\|^2
+
\frac{\delta^{3/2}\mu}{2}\|x-x^\star\|^2
\\
&\leq
\sqrt{\delta}\,\frac{\mu}{2}\|y-x^\star\|^2
+
\sqrt{\delta}\,D_{f_{-\hat\mu}}(x^\star,x).
\end{aligned}
$$
Therefore,
$$
\begin{aligned}
-\big\langle \nabla \mathcal E(\bs z;\hat\mu),\mathcal G(\bs z)\big\rangle
\geq{}&
\langle \nabla f_{-\hat\mu}(x),x-x^\star\rangle
-
\sqrt{\delta}\,D_{f_{-\hat\mu}}(x^\star,x)
+
\beta\|\nabla f_{-\hat\mu}(x)\|^2
\\
&\quad
+
\left(2-\sqrt{\delta}\right)\frac{\mu}{2}\|y-x^\star\|^2
+
\beta\hat\mu\langle \nabla f_{-\hat\mu}(x),x-x^\star\rangle.
\end{aligned}
$$
By the symmetrization identity \eqref{eq:symBre} and the invariance of the Bregman asymmetry under quadratic shifts,
$D_{f_{-\hat\mu}}(x^\star,x)
=
D_{f_{-\hat\mu}}(x,x^\star)
+
\Delta_f(x,x^\star).$
Hence
$$
\begin{aligned}
&\langle \nabla f_{-\hat\mu}(x),x-x^\star\rangle
-
\sqrt{\delta}\,D_{f_{-\hat\mu}}(x^\star,x)
\\
={}&
D_{f_{-\hat\mu}}(x,x^\star)
+
\left(1-\sqrt{\delta}\right)D_{f_{-\hat\mu}}(x^\star,x)
\\
={}&
\left(2-\sqrt{\delta}\right)D_{f_{-\hat\mu}}(x,x^\star)
+
\left(1-\sqrt{\delta}\right)\Delta_f(x,x^\star).
\end{aligned}
$$
Combining the first term with $\frac{\mu}{2}\|y-x^\star\|^2$ yields
$$
\begin{aligned}
-\big\langle \nabla \mathcal E(\bs z;\hat\mu),\mathcal G(\bs z)\big\rangle
\geq{}&
\left(2-\sqrt{\delta}\right)\mathcal E(\bs z;\hat\mu)
+
\beta\|\nabla f_{-\hat\mu}(x)\|^2
\\
&\quad
+
\beta\hat\mu\langle \nabla f_{-\hat\mu}(x),x-x^\star\rangle
+
\left(1-\sqrt{\delta}\right)\Delta_f(x,x^\star).
\end{aligned}
$$
Finally,
$$
\left(2-\sqrt{\delta}\right)\mathcal E(\bs z;\hat\mu)
=
\left(2-\sqrt{\delta}-b\right)\mathcal E(\bs z;\hat\mu)
+
b\mathcal E(\bs z;\hat\mu).
$$
Since $b\geq0$, 
using
$$
b\mathcal E(\bs z;\hat\mu)
\geq
bD_{f_{-\hat\mu}}(x,x^\star),
$$
we obtain \eqref{eq:energy_identity}.
\end{proof}

Using the refined strong Lyapunov property in Lemma~\ref{lem:energy_identity}, we obtain the following one-step inequality with perturbation terms. \RV{The proof follows Lemma~\ref{lm:decay1-HNAGplus}, retaining only the leading-order terms in $\alpha$.}

\begin{lemma}\label{lm:decay1}
Let $0\leq\hat\mu\leq\mu$, and let $(x_k,y_k)$ be generated by the HNAG scheme~\eqref{eq:HNAG} with
$\alpha=\sqrt{2\mu/L}$ and $\alpha\beta=1/L$.
Then, for all $k\geq0$,
\begin{equation}\label{eq:decay1}
\begin{aligned}
(1+c_L\alpha)\mathcal E(\bs z_{k+1};\hat\mu)
\leq{}&
\mathcal E(\bs z_k;\hat\mu)
+\frac1{2L}\|\nabla f_{-\hat\mu}(x_{k+1})\|^2
-\frac1{2L}\|\nabla f_{-\hat\mu}(x_k)\|^2
\\
&\quad
-\alpha\left[
d_\delta\Delta_f(x_{k+1},x^\star)
+\RV{bD_{f_{-\hat\mu}}(x_{k+1},x^\star)}
\right ],
\end{aligned}
\end{equation}
where
$$
b\geq0,\quad
\delta=\frac{\mu-\hat\mu}{\mu}\in[0,1],\quad
c_L=2-\sqrt{\delta}-b>0,\quad
d_\delta=1-\sqrt{\delta}.
$$
\end{lemma}

\begin{proof}
Using the Bregman expansion,
$$
\mathcal E(\bs z_{k+1};\hat\mu)-\mathcal E(\bs z_k;\hat\mu)
=
\big\langle \nabla \mathcal E(\bs z_{k+1};\hat\mu),\bs z_{k+1}-\bs z_k\big\rangle
-
D_{\mathcal E(\cdot;\hat\mu)}(\bs z_k,\bs z_{k+1}).
$$
Let $g_k:=\nabla f_{-\hat\mu}(x_k)$. \RV{Since $-\beta\nabla f(x_{k+1})$ is used in $\mathcal G(\bs z_{k+1})$ and
$$
\beta(g_{k+1}-g_k)+\beta\hat\mu(x_{k+1}-x_k)
=
\beta\bigl(\nabla f(x_{k+1})-\nabla f(x_k)\bigr),
$$}
the update can be written as
$$
\bs z_{k+1}-\bs z_k
=
\alpha\mathcal G(\bs z_{k+1})
+
\alpha
\binom{
(y_k-y_{k+1})+\beta(g_{k+1}-g_k)+\beta\hat\mu(x_{k+1}-x_k)
}{0}.
$$

The implicit Euler part is controlled by Lemma~\ref{lem:energy_identity}:
\begin{equation}\label{eq:IE-HNAGpp}
\begin{aligned}
&\alpha\big\langle \nabla \mathcal E(\bs z_{k+1};\hat\mu),\mathcal G(\bs z_{k+1})\big\rangle \\
\leq{}&
-c_L\alpha\mathcal E(\bs z_{k+1};\hat\mu)
-\alpha\beta\hat\mu\langle g_{k+1},x_{k+1}-x^\star\rangle
\\
&\quad
-\alpha\beta\|g_{k+1}\|^2
-\alpha\left [
d_\delta\Delta_f(x_{k+1},x^\star)
+\RV{bD_{f_{-\hat\mu}}(x_{k+1},x^\star)}
\right ].
\end{aligned}
\end{equation}

\medskip
\noindent
\textbf{The $\hat\mu$-shift term.}
By the three-point Bregman identity~\eqref{eq:Bregmanidentity},
$$
\begin{aligned}
&\alpha\beta\hat\mu\langle g_{k+1},x_{k+1}-x_k\rangle
\\
={}&
\alpha\beta\hat\mu\Bigl(
D_{f_{-\hat\mu}}(x_{k+1},x^\star)
+
D_{f_{-\hat\mu}}(x_k,x_{k+1})
-
D_{f_{-\hat\mu}}(x_k,x^\star)
\Bigr).
\end{aligned}
$$
The first term cancels the corresponding part of the shift term in \eqref{eq:IE-HNAGpp}, leaving a nonpositive term, which we may discard.
The second term combines with
$-D_{\mathcal E(\cdot;\hat\mu)}(\bs z_k,\bs z_{k+1})$ to give
$$
-(1-\alpha\beta\hat\mu)D_{f_{-\hat\mu}}(x_k,x_{k+1}).
$$
The last term is also nonpositive and may be discarded.

\medskip
\noindent
\textbf{Lagging of $y$.}
By Cauchy--Schwarz and Young's inequality,
$$
\alpha\langle g_{k+1},y_k-y_{k+1}\rangle
\leq
\frac{\alpha^2}{2\mu}\|g_{k+1}\|^2
+
\frac{\mu}{2}\|y_k-y_{k+1}\|^2.
$$
The second term cancels the $y$-component of
$-D_{\mathcal E(\cdot;\hat\mu)}(\bs z_k,\bs z_{k+1})$.

\medskip
\noindent
\textbf{Lagging of the gradient.}
The identity of squares gives
$$
\begin{aligned}
\alpha\beta\langle g_{k+1},g_{k+1}-g_k\rangle
={}&
-\frac{\alpha\beta}{2}\|g_k\|^2
+\frac{\alpha\beta}{2}\|g_{k+1}\|^2
+\frac{\alpha\beta}{2}\|g_{k+1}-g_k\|^2.
\end{aligned}
$$
Since $f_{-\hat\mu}$ is $(L-\hat\mu)$-smooth and $\alpha\beta=1/L$,
$$
\frac{\alpha\beta}{2}\|g_{k+1}-g_k\|^2
\leq
(1-\alpha\beta\hat\mu)D_{f_{-\hat\mu}}(x_k,x_{k+1}).
$$
Therefore,
$$
\begin{aligned}
&\alpha\beta\langle g_{k+1},g_{k+1}-g_k\rangle
-
(1-\alpha\beta\hat\mu)D_{f_{-\hat\mu}}(x_k,x_{k+1})
\\
\leq{}&
-\frac{\alpha\beta}{2}\|g_k\|^2
+
\frac{\alpha\beta}{2}\|g_{k+1}\|^2.
\end{aligned}
$$

Combining the above estimates gives
$$
\begin{aligned}
(1+c_L\alpha)\mathcal E(\bs z_{k+1};\hat\mu)
\leq{}&
\mathcal E(\bs z_k;\hat\mu)
+
\left(
\frac{\alpha^2}{2\mu}
-\frac{\alpha\beta}{2}
\right)\|g_{k+1}\|^2
-
\frac{\alpha\beta}{2}\|g_k\|^2
\\
&\quad
-\alpha bD_{f_{-\hat\mu}}(x_{k+1},x^\star)
-\alpha d_\delta\Delta_f(x_{k+1},x^\star).
\end{aligned}
$$
Finally, using $\alpha^2=2\mu/L$ and $\alpha\beta=1/L$, we have
the desired result \eqref{eq:decay1}. $\Box$
\end{proof}


{We record two limiting cases of Lemma~\ref{lm:decay1}. They clarify the role of the shift parameter \(\hat\mu=(1-\delta)\mu\).}

\paragraph{\bf Case 1: \(\delta=1\), no primal shift.} In this case, \(\hat\mu=0\) and \(f_{-\hat\mu}=f\). Since \(d_\delta=0\), there is no \(\Delta_f(x_{k+1},x^{\star})\) term. Taking \(b=0\), we have
$$
c_L=1,
$$
which yields the suboptimal rate \((1+\sqrt{2/\kappa})^{-1}\) in Theorem~\ref{thm:linearconvHNAG-funcval}.

\paragraph{\bf Case 2: \(\delta=0\), full primal shift.}
In this case, \(\hat\mu=\mu\), \(d_\delta=1\), and, taking \(b=0\), we have the full coercivity
$$
c_L=2.
$$
\begin{proposition}[Full primal shift]\label{prop:full-shift}
Assume $\delta=0$, so that $\hat\mu=\mu$, $d_\delta=1$, and $b=0$. Let
$$
\widetilde{\mathcal E}_k
:=
\mathcal E(\bs z_k;\mu)
-\frac1{2L}\|\nabla f_{-\mu}(x_k)\|^2.
$$
Then
$$
\widetilde{\mathcal E}_{k+1}
\le
\frac1{1+2\alpha}\widetilde{\mathcal E}_k
-\frac{\alpha}{1+2\alpha}\Delta_f(x_{k+1},x^\star).
$$
If $\Delta_f(x_j,x^\star)\geq0$ for all $j$, \RV{which holds when $f$ is quadratic}, then
\begin{equation}\label{eq:Delta=0}
\widetilde{\mathcal E}_k
\leq
\widetilde{\mathcal E}_0
\left(\frac1{1+2\sqrt{2/\kappa}}\right)^k.
\end{equation}
More generally, assume $\alpha\leq1$. If
$$
\Delta_f(x_j,x^\star)
\geq
-C_\Delta\left(\frac{1-\alpha}{1+2\alpha}\right)^j,
\qquad j\geq0,
$$
then
\begin{equation}\label{eq:highorder}
\widetilde{\mathcal E}_k
\leq
(\widetilde{\mathcal E}_0+C_\Delta)
\left(\frac1{1+2\sqrt{2/\kappa}}\right)^k.
\end{equation}
\end{proposition}

\begin{proof}
Lemma~\ref{lm:decay1} gives
$$
\begin{aligned}
(1+2\alpha)\mathcal E(\bs z_{k+1};\mu)
\leq{}&
\mathcal E(\bs z_k;\mu)
+\frac1{2L}\|\nabla f_{-\mu}(x_{k+1})\|^2
-\frac1{2L}\|\nabla f_{-\mu}(x_k)\|^2
\\
&\quad
-\alpha\Delta_f(x_{k+1},x^\star).
\end{aligned}
$$
Let $r=(1+2\alpha)^{-1}$. By the definition of $\widetilde{\mathcal E}_k$, subtracting $\frac{1+2\alpha}{2L}\|\nabla f_{-\mu}(x_{k+1})\|^2$ from both sides gives
$$
\begin{aligned}
(1+2\alpha)\widetilde{\mathcal E}_{k+1}
\leq{}&
\widetilde{\mathcal E}_k
-\frac{\alpha}{L}\|\nabla f_{-\mu}(x_{k+1})\|^2
-\alpha\Delta_f(x_{k+1},x^\star)
\\
\leq{}&
\widetilde{\mathcal E}_k
-\alpha\Delta_f(x_{k+1},x^\star).
\end{aligned}
$$
Therefore,
$$
\widetilde{\mathcal E}_{k+1}
\leq
r\widetilde{\mathcal E}_k
-\alpha r\Delta_f(x_{k+1},x^\star).
$$
Iterating yields
$$
\widetilde{\mathcal E}_k
\leq
r^k\widetilde{\mathcal E}_0
-\alpha r\sum_{j=1}^k r^{k-j}\Delta_f(x_j,x^\star).
$$
If $\Delta_f(x_j,x^\star)\geq0$, then the perturbation term is nonpositive, and \eqref{eq:Delta=0} follows.

For the second claim, use
$$
-\Delta_f(x_j,x^\star)
\leq
C_\Delta\left(\frac{1-\alpha}{1+2\alpha}\right)^j
=
C_\Delta \, r^j(1-\alpha)^j.
$$
Then
$$
\begin{aligned}
\widetilde{\mathcal E}_k
&\leq
r^k\widetilde{\mathcal E}_0
+
\alpha r C_\Delta
\sum_{j=1}^k r^{k-j}r^j(1-\alpha)^j
=
r^k\widetilde{\mathcal E}_0
+
\alpha r C_\Delta\, r^k
\sum_{j=1}^k(1-\alpha)^j.
\\
&\leq
r^k\widetilde{\mathcal E}_0
+
r(1-\alpha)C_\Delta \, r^k \leq
(\widetilde{\mathcal E}_0+C_\Delta)r^k.
\end{aligned}
$$
This proves \eqref{eq:highorder}. $\Box$
\end{proof}

Thus the full-shift argument yields the leading rate $1-2\sqrt{2/\kappa}$ when the Bregman asymmetry is nonnegative or decays faster than the target contraction factor. The sign of $\Delta_f(x_j,x^\star)$ may vary along the iteration, so the weighted sum $-\alpha r\sum_{j=1}^k r^{k-j}\Delta_f(x_j,x^\star)$ may also exhibit cancellation. Since its sign is hard to control, we introduce a new function class and bound its magnitude directly.

\subsection{\RV{A new function class}}\label{sec:LAS}

The full-shift estimate shows that the relevant quantity is the Bregman asymmetry relative to the Bregman divergence. This motivates the following definition.

\begin{definition}[Local asymptotic symmetry]
Let $f\in\mathcal S_{\mu,L}$ with minimizer $x^\star$. The function $f$ is said to be \emph{locally asymptotically symmetric at $x^\star$}, abbreviated \emph{LAS at $x^\star$}, if
$$
\lim_{x\to x^\star}
\RV{\frac{|\Delta_f(x,x^\star)|}{D_f(x,x^\star)}}
=
0.
$$
Define
$$
\omega(R)
:=
\sup_{0<\|x-x^\star\|\leq R}
\frac{|\Delta_f(x,x^\star)|}{D_f(x,x^\star)},
\qquad R>0.
$$
Then $\omega$ is monotone nondecreasing, and the LAS condition is equivalent to
$$
\lim_{R\to0^+}\omega(R)=0.
$$
\end{definition}

We define
$$
\mathcal S_{\mu,L}^{\mathrm{LAS}}(x^\star)
:=
\Bigl\{
f\in\mathcal S_{\mu,L}:
f\text{ is LAS at }x^\star
\Bigr\}.
$$

\RV{For $f\in\mathcal S_{\mu,L}$, using \eqref{eq:muLsquares}, we have
$$
\frac{2}{L}
\frac{|\Delta_f(x,x^\star)|}{\|x-x^\star\|^2}
\leq
\frac{|\Delta_f(x,x^\star)|}{D_f(x,x^\star)}
\leq
\frac{2}{\mu}
\frac{|\Delta_f(x,x^\star)|}{\|x-x^\star\|^2}.
$$
Thus the ratio-form LAS condition is equivalent to
$$
\frac{|\Delta_f(x,x^\star)|}{\|x-x^\star\|^2}
\to0
\qquad\text{as }x\to x^\star.
$$
The ratio form $|\Delta_f(x,x^\star)|/D_f(x,x^\star)$ is more natural for the Bregman geometry and simplifies the asymptotic analysis below.}

The LAS condition means that, near the minimizer, the Bregman asymmetry $\Delta_f(x,x^\star)$ is negligible relative to the local Bregman energy. It is weaker than $\mathcal C^2$ regularity, but is sufficient to treat the Bregman asymmetry as a higher-order perturbation.

The following proposition gives a practical sufficient condition for LAS in terms of directional second derivatives.

\begin{proposition}[Uniform directional $\mathcal C^2$ regularity implies LAS]
\label{prop:directional-C2-LAS}
Assume $f\in\mathcal S_{\mu,L}$. Suppose that, for some $R_0>0$, each ray function
$$
\phi_\theta(r):=f(x^\star+r\theta),
\qquad
\theta\in\mathbb S^{d-1} \RV{= \{x\in \mathbb R^d: \|x\| = 1\}},
$$
belongs to $\mathcal C^2(0,R_0)$, and that the directional second derivatives have vanishing uniform oscillation near the minimizer, namely,
$$
\omega_2(R)
:=
\sup_{\theta\in\mathbb S^{d-1}}
\sup_{0<s,t\leq R}
|\phi_\theta''(s)-\phi_\theta''(t)|
\to0
\qquad\text{as }R\to0^+.
$$
Then $f\in\mathcal S_{\mu,L}^{\mathrm{LAS}}(x^\star)$. More precisely, for $0<R\leq R_0$ and $\|x-x^\star\|\leq R$,
$$
|\Delta_f(x,x^\star)|
\leq
\frac12\omega_2(R)\|x-x^\star\|^2\leq \frac{\omega_2(R)}{\mu}D_f(x,x^\star).
$$
\end{proposition}

\begin{proof}
The assumption on $\omega_2$ implies that, for each $\theta\in\mathbb S^{d-1}$, $\phi_\theta''(r)$ has a limit as $r\to0^+$. We extend $\phi_\theta''$ continuously to $r=0$ by this limit.

Let $x=x^\star+r\theta$, where $r=\|x-x^\star\|$ and $\theta\in\mathbb S^{d-1}$. Since $x^\star$ is a minimizer,
$
\phi_\theta'(0)
=
0.
$
The $L$-smoothness of $f$ implies that $\phi_\theta'$ is Lipschitz. Therefore,
$$
\phi_\theta(r)-\phi_\theta(0)
=
\int_0^r(r-s)\phi_\theta''(s)\dd s,
\qquad
\phi_\theta'(r)
=
\int_0^r\phi_\theta''(s)\dd s.
$$
Hence
$$
\begin{aligned}
\Delta_f(x,x^\star)
&=
D_f(x^\star,x)-D_f(x,x^\star)
\\
&=
r\phi_\theta'(r)
-
2\bigl(\phi_\theta(r)-\phi_\theta(0)\bigr)
\\
&=
r^2\int_0^1(2t-1)\phi_\theta''(tr)\dd t.
\end{aligned}
$$
Since $\int_0^1(2t-1)\dd t=0$,
we may subtract $\phi_\theta''(0)$ to obtain
$$
\begin{aligned}
|\Delta_f(x,x^\star)|
&\leq
r^2\int_0^1|2t-1|
\bigl|\phi_\theta''(tr)-\phi_\theta''(0)\bigr|\dd t
\\
&\leq
r^2\omega_2(R)\int_0^1|2t-1|\dd t
=
\frac12\omega_2(R)r^2.
\end{aligned}
$$
Since $f$ is $\mu$-strongly convex,
$D_f(x,x^\star)\geq\frac{\mu}{2}r^2$.
Thus
$$
\frac{|\Delta_f(x,x^\star)|}{D_f(x,x^\star)}
\leq
\frac{\omega_2(R)}{\mu}
\to0
\qquad\text{as }R\to0^+.
$$
This proves the LAS condition. $\Box$
\end{proof}

In particular,
$$
\mathcal C^2\cap\mathcal S_{\mu,L}
\subset
\mathcal S_{\mu,L}^{\mathrm{LAS}}(x^\star).
$$
Indeed, if $\nabla^2 f$ is continuous near $x^\star$, then
$$
\phi_\theta''(r)
=
\left\langle
\nabla^2 f(x^\star+r\theta)\theta,\theta
\right\rangle.
$$
The uniform continuity of $\nabla^2 f$ on a compact neighborhood of $x^\star$ implies that $\omega_2(R)\to0$ as $R\to0^+$.

However, $\mathcal C^2$ regularity is not necessary. As indicated by Proposition~\ref{prop:directional-C2-LAS}, piecewise $\mathcal C^2$ functions may still belong to $\mathcal S_{\mu,L}^{\mathrm{LAS}}(x^\star)$. A simple example is
$$
f(x)=
\begin{cases}
\tfrac{L}{2}x^2, & x\geq0,\\[1mm]
\tfrac{\mu}{2}x^2, & x<0,
\end{cases}
\qquad 0<\mu<L.
$$
Then $f\in\mathcal C^1(\mathbb R)\cap\mathcal S_{\mu,L}$, while $f''$ is piecewise constant and discontinuous at $x^\star=0$. Since $f$ is quadratic on each side of $0$,
$\Delta_f(x,0)=0$ for all $x$ so the LAS condition holds trivially.

\RV{To obtain a nonzero Bregman asymmetry, we add a smooth perturbation that preserves the minimizer. Let}
\begin{equation}\label{eq:LASexample}
f_\epsilon(x)=f(x)+\epsilon\sin^2x,
\qquad
\epsilon=\frac{\mu}{4}.
\end{equation}
Then $f_\epsilon\in\mathcal S_{\mu/2,L+\mu/2}$ and $x^\star = \arg\min f_\epsilon = 0$.
For $h(x)=\epsilon\sin^2x$,
$$
\begin{aligned}
\Delta_h(x,0)
&=
-2\epsilon\sin^2x+\epsilon x\sin(2x)
=
-\frac{2\epsilon}{3}x^4+O(x^6).
\end{aligned}
$$
The piecewise quadratic part has zero Bregman asymmetry, and hence
$$
|\Delta_{f_\epsilon}(x,0)|
=
O(|x|^4), \quad 
D_{f_\epsilon}(x,0)
\geq
\frac{\mu}{4}|x|^2.
$$
Therefore,
$$
\frac{|\Delta_{f_\epsilon}(x,0)|}{D_{f_\epsilon}(x,0)}
=
O(|x|^2)
\to0
\qquad\text{as }x\to0.
$$
Thus the function \eqref{eq:LASexample} belongs to $\mathcal S_{\mu/2,L+\mu/2}^{\mathrm{LAS}}(0)$ and has a nonzero but asymptotically negligible Bregman asymmetry.

\subsection{A sequence of partial shifts}

Let \(\{\mu_k\}\) be a nondecreasing sequence with \(\mu_k\uparrow\mu\). Define
\begin{equation*}
\mathcal E(\bs z_k;\mu_k)
:=
D_{f_k}(x_k,x^\star)+\frac{\mu}{2}\|y_k-x^\star\|^2,
\qquad
f_k:=f_{-\mu_k}.
\end{equation*}
The parameter \(\mu_k\) is used only in the shift of \(f_k\); the quadratic term in \(y_k\) always uses the fixed coefficient \(\mu\). We write
$$
\delta_k:=\frac{\mu-\mu_k}{\mu},
\qquad
\Delta\delta_k:=\delta_k-\delta_{k+1}\ge0.
$$
Changing the shift also changes the gradient term. The resulting difference is a higher-order perturbation.

\begin{lemma}\label{lm:gradient_norm_ineq}
For all \(k\ge 0\),
$$
\|\nabla f_k(x_{k+1})\|^2
\le
\|\nabla f_{k+1}(x_{k+1})\|^2
+
L\Delta\delta_k\alpha^2
\langle \nabla f_k(x_{k+1}),x_{k+1}-x^\star\rangle.
$$
\end{lemma}
\begin{proof}
By definition,
$$
\nabla f_{k+1}(x_{k+1})
=
\nabla f_k(x_{k+1})-(\mu_{k+1}-\mu_k)(x_{k+1}-x^\star).
$$
Expanding the square and dropping the positive last term, we obtain
$$
\begin{aligned}
\|\nabla f_{k+1}(x_{k+1})\|^2
&=
\bigl\|\nabla f_k(x_{k+1})-(\mu_{k+1}-\mu_k)(x_{k+1}-x^\star)\bigr\|^2\\
&\ge
\|\nabla f_k(x_{k+1})\|^2
-
2(\mu_{k+1}-\mu_k)
\langle \nabla f_k(x_{k+1}),x_{k+1}-x^\star\rangle.
\end{aligned}
$$
Since $\mu_{k+1}-\mu_k=\Delta\delta_k\,\mu, 
\alpha^2=2\mu/L,$
we have
$$
2(\mu_{k+1}-\mu_k)=L\Delta\delta_k\alpha^2.
$$
This proves the claim.
$\Box$
\end{proof}

By Theorem~\ref{thm:linearconvHNAG-funcval},
$$
\|x_{k+1}-x^\star\|
\leq
C_1\left(1+\sqrt{2\rho}\right)^{-\frac{k+1}{2}},
\qquad
C_1:=
\left(
\frac{2(1+2\alpha)}{\alpha\mu}\widetilde E_0
\right)^{1/2}.
$$
If $f\in\mathcal S_{\mu,L}^{\rm LAS}(x^\star)$, then
\begin{equation}\label{eq:def-omegak}
\omega(\|x_{k+1}-x^\star\|)
\leq
\omega\left(
C_1\left(1+\sqrt{2\rho}\right)^{-\frac{k+1}{2}}
\right)
=:\omega_k.
\end{equation}
Since $\omega$ is nondecreasing and its argument decreases to zero, $\{\omega_k\}$ is nonincreasing and converges to zero.

The proof of the following one-step contraction has three steps. First, we apply Lemma~\ref{lm:decay1} with the partial shift $\mu_k$. Second, we use Lemma~\ref{lm:gradient_norm_ineq} to control the higher-order perturbation caused by changing the shifted gradient from $\nabla f_k$ to $\nabla f_{k+1}$. Third, we change the energy parameter from $\mu_k$ to $\mu_{k+1}$ and absorb the remaining perturbation using \eqref{eq:omega_condition}.

\begin{lemma}\label{lm:shifted_energy_recursion}
Assume $f\in\mathcal S_{\mu,L}^{\rm LAS}(x^\star)$. Let $(x_k,y_k)$ be generated by Algorithm~\ref{alg:HNAGpp}. Let $\{\delta_k\}\subset[0,1]$ be nonincreasing with $\delta_k\to0$, and let $\{b_k\}$ be nonnegative with $b_k\to0$. Set
$$
\mu_k=(1-\delta_k)\mu,\quad
f_k:=f_{-\mu_k},\quad
d_k:=1-\sqrt{\delta_k},\quad
\Delta\delta_k:=\delta_k-\delta_{k+1}\geq0,
$$
and define
$$
\widetilde{\mathcal E}(\bs z_k;\mu_k)
:=
\mathcal E(\bs z_k;\mu_k)
-\frac1{2L}\|\nabla f_k(x_k)\|^2\geq 0.
$$
Let $\omega_k$ be defined by \eqref{eq:def-omegak}, and set
$$
c_k:=2-\sqrt{\delta_k}-b_k\geq0,
\qquad
r_k:=
\left(1+c_k\alpha-\Delta\delta_k\alpha^2\right)^{-1}.
$$
Assume $r_k^{-1}\geq1$ and
\begin{equation}\label{eq:omega_condition}
\left|
d_k-\frac{\Delta\delta_k\alpha}{2}
\right|\omega_k
\leq
b_k\delta_k.
\end{equation}
Then
$$
\widetilde{\mathcal E}(\bs z_{k+1};\mu_{k+1})
\leq
r_k\widetilde{\mathcal E}(\bs z_k;\mu_k).
$$
\end{lemma}

\begin{proof}
Apply Lemma~\ref{lm:decay1} with $\hat\mu=\mu_k$. Then
$$
\begin{aligned}
(1+c_k\alpha)\mathcal E(\bs z_{k+1};\mu_k)
\leq{}&
\mathcal E(\bs z_k;\mu_k)
+\frac1{2L}\|\nabla f_k(x_{k+1})\|^2
-\frac1{2L}\|\nabla f_k(x_k)\|^2
\\
&\quad
-\alpha\left(
d_k\Delta_f(x_{k+1},x^\star)
+
b_kD_{f_k}(x_{k+1},x^\star)
\right).
\end{aligned}
$$
We first change the shifted gradient using Lemma~\ref{lm:gradient_norm_ineq}:
$$
\frac1{2L}\|\nabla f_k(x_{k+1})\|^2
\leq
\frac1{2L}\|\nabla f_{k+1}(x_{k+1})\|^2
+
\frac{\Delta\delta_k\alpha^2}{2}
\langle\nabla f_k(x_{k+1}),x_{k+1}-x^\star\rangle.
$$
Using
$$
\begin{aligned}
\frac{\Delta\delta_k\alpha^2}{2}
\langle\nabla f_k(x_{k+1}),x_{k+1}-x^\star\rangle
={}&
\Delta\delta_k\alpha^2D_{f_k}(x_{k+1},x^\star)
+\frac{\Delta\delta_k\alpha^2}{2}
\Delta_f(x_{k+1},x^\star),
\end{aligned}
$$
and
$$
\Delta\delta_k\alpha^2D_{f_k}(x_{k+1},x^\star)
\leq
\Delta\delta_k\alpha^2\mathcal E(\bs z_{k+1};\mu_k),
$$
the first term can be absorbed into the left-hand side, which slightly weakens the contraction factor from $1+c_k\alpha$ to $r_k$. 
We then obtain
$$
\begin{aligned}
r_k^{-1}\mathcal E(\bs z_{k+1};\mu_k)
\leq{}&
\mathcal E(\bs z_k;\mu_k)
+\frac1{2L}\|\nabla f_{k+1}(x_{k+1})\|^2
-\frac1{2L}\|\nabla f_k(x_k)\|^2
\\
&\quad
-\alpha\left[
\left(
d_k-\frac{\Delta\delta_k\alpha}{2}
\right)\Delta_f(x_{k+1},x^\star)
+
b_kD_{f_k}(x_{k+1},x^\star)
\right].
\end{aligned}
$$
Since $\mu_{k+1}\geq\mu_k$,
$$
\mathcal E(\bs z_{k+1};\mu_{k+1})
\leq
\mathcal E(\bs z_{k+1};\mu_k).
$$
Thus the same inequality holds with $\mathcal E(\bs z_{k+1};\mu_{k+1})$ on the left-hand side.

By the definition of $\omega_k$,
$$
|\Delta_f(x_{k+1},x^\star)|
\leq
\omega_kD_f(x_{k+1},x^\star).
$$
Moreover, since $\mu_k=(1-\delta_k)\mu$ and $f$ is $\mu$-strongly convex,
$$
\begin{aligned}
D_{f_k}(x_{k+1},x^\star)
&=
D_f(x_{k+1},x^\star)
-\frac{\mu_k}{2}\|x_{k+1}-x^\star\|^2
\\
&\geq
D_f(x_{k+1},x^\star)
-\frac{\mu_k}{\mu}D_f(x_{k+1},x^\star)
=
\delta_kD_f(x_{k+1},x^\star).
\end{aligned}
$$
Therefore,
$$
\begin{aligned}
&
\left(
d_k-\frac{\Delta\delta_k\alpha}{2}
\right)\Delta_f(x_{k+1},x^\star)
+
b_kD_{f_k}(x_{k+1},x^\star)
\\
&\geq
\left[
b_k\delta_k
-
\left|
d_k-\frac{\Delta\delta_k\alpha}{2}
\right|\omega_k
\right]
D_f(x_{k+1},x^\star)
\geq0,
\end{aligned}
$$
where the last inequality follows from \eqref{eq:omega_condition}. It follows that
$$
\begin{aligned}
r_k^{-1}\mathcal E(\bs z_{k+1};\mu_{k+1})
\leq{}&
\mathcal E(\bs z_k;\mu_k)
+\frac1{2L}\|\nabla f_{k+1}(x_{k+1})\|^2
-\frac1{2L}\|\nabla f_k(x_k)\|^2.
\end{aligned}
$$
Subtracting $r_k^{-1}(2L)^{-1}\|\nabla f_{k+1}(x_{k+1})\|^2$ from both sides gives
$$
\begin{aligned}
r_k^{-1}\widetilde{\mathcal E}(\bs z_{k+1};\mu_{k+1})
\leq{}&
\widetilde{\mathcal E}(\bs z_k;\mu_k)
+
\frac{1-r_k^{-1}}{2L}
\|\nabla f_{k+1}(x_{k+1})\|^2.
\end{aligned}
$$
Since $r_k^{-1}\geq1$, the last term is nonpositive. Therefore,
$$
r_k^{-1}\widetilde{\mathcal E}(\bs z_{k+1};\mu_{k+1})
\leq
\widetilde{\mathcal E}(\bs z_k;\mu_k),
$$
which proves the claim.
\end{proof}

\subsection{Asymptotic Convergence Analysis}

We now prove the asymptotic rate $1-2\sqrt{2/\kappa}$ for HNAG$^{++}$ under the LAS assumption. For the finitely many initial steps, we use the unshifted estimate. Once the LAS modulus $\omega_k$ is small, we introduce partial shifts $\mu_k\uparrow\mu$ so that the asymmetry perturbation is absorbed by the positive shifted-Bregman term.

\begin{theorem}[Asymptotic rate $1-2\sqrt{2/\kappa}$]\label{thm:main-thm-C21}
Assume $f\in\mathcal S_{\mu,L}^{\rm LAS}(x^\star)$. Let $(x_k,y_k)$ be generated by Algorithm~\ref{alg:HNAGpp}, let $\omega_k$ be defined by \eqref{eq:def-omegak}, and set $K:=\min\{k\geq0:\omega_k\leq1/8\}$. Then there exist nonnegative sequences $\{\mu_k\}$, $\{\delta_k\}$, $\{b_k\}$, and $\{c_k\}$, and a constant $C>0$, such that
$$
\mu_k=(1-\delta_k)\mu,\qquad
\delta_k\downarrow0,\qquad
b_k\to0,\qquad
c_k=2-\sqrt{\delta_k}-b_k\to2,
$$
and
\RV{
$$
\frac{\mu}{2}\|y_{k+1}-x^\star\|^2
\leq
\widetilde{\mathcal E}(\bs z_{k+1};\mu_{k+1})
\leq
C\prod_{i=0}^{k}r_i.
$$}
Here
$$
\widetilde{\mathcal E}(\bs z_k;\mu_k)
:=
\mathcal E(\bs z_k;\mu_k)
-\frac1{2L}\|\nabla f_k(x_k)\|^2,
\qquad
f_k:=f_{-\mu_k},
$$
and
$$
r_k:=
\begin{cases}
(1+\alpha)^{-1}, & 0\leq k<K,\\[1mm]
\left(1+c_k\alpha-\Delta\delta_k\alpha^2\right)^{-1}, & k\geq K,
\end{cases}
\quad
\alpha=\sqrt{\frac{2}{\kappa}},
\quad
\Delta\delta_k:=\delta_k-\delta_{k+1}.
$$
Consequently, $\widetilde{\mathcal E}(\bs z_k;\mu_k)$ \RV{and $\|y_k-x^\star\|^2$} converge $R$-linearly with asymptotic rate
$$
r_\ast
=
\frac1{1+2\sqrt{2/\kappa}}
=
1-2\sqrt{2/\kappa}+O(\kappa^{-1})
\qquad\text{as }\kappa\to\infty.
$$
\end{theorem}

\begin{proof}
Since $f\in\mathcal S_{\mu,L}^{\rm LAS}(x^\star)$, we have $\omega_k\downarrow0$, so the index $K$ in the theorem statement is finite. The threshold $1/8$ is only a convenient choice ensuring $q_k\leq1/2$; any fixed sufficiently small threshold would work.

\medskip
\noindent
\textbf{Step 1: Initial finite segment.}
For $0\leq k<K$, we use the unshifted choice
$$
\delta_k=1,\qquad
\mu_k=0,\qquad
b_k=0,\qquad
d_k=0,\qquad
c_k=1.
$$
Then
$$
\widetilde{\mathcal E}(\bs z_{k+1};0)
\leq
(1+\alpha)^{-1}
\widetilde{\mathcal E}(\bs z_k;0),
\qquad
0\leq k<K.
$$

\medskip
\noindent
\textbf{Step 2: Choice of the partial shifts.}
For $k\geq K$, set
$$
q_k:=\omega_k^{1/3},\qquad
\delta_k:=q_k^2,\qquad
d_k:=1-q_k,\qquad
b_k:=d_kq_k.
$$
This choice gives
$$
b_k\delta_k=d_k\omega_k.
$$
Since $\omega_k\leq1/8$, we have $0\leq q_k\leq1/2$, and hence $d_k,b_k\geq0$. Moreover,
$$
\begin{aligned}
c_k
&:=
2-\sqrt{\delta_k}-b_k
=
1+(1-q_k)^2
\geq\frac54,\\
&q_k\to0,\qquad
\delta_k\downarrow0,\qquad
b_k\to0,\qquad
c_k\to2.
\end{aligned}
$$

\medskip
\noindent
\textbf{Step 3: Verification of Lemma~\ref{lm:shifted_energy_recursion}.}
Set
$$
r_k
:=
\left(1+c_k\alpha-\Delta\delta_k\alpha^2\right)^{-1}.
$$
Since
$$
c_k\geq\frac54,
\qquad
0\leq\Delta\delta_k\leq\delta_k=q_k^2\leq\frac14,
$$
and $\kappa=L/\mu\geq1$ implies $\alpha=\sqrt{2/\kappa}\leq\sqrt2$, we have
\begin{equation*}
\begin{aligned}
r_k^{-1}
&=
1+c_k\alpha-\Delta\delta_k\alpha^2
\geq
1+\frac54\alpha-\frac14\alpha^2
=
1+\frac{\alpha}{4}(5-\alpha)
>1.
\end{aligned}
\end{equation*}

It remains to verify \eqref{eq:omega_condition}. Since $d_k\geq1/2$,
$$
d_k-\frac{\Delta\delta_k\alpha}{2}
\geq
\frac12-\frac{\sqrt2}{8}
>0.
$$
Therefore,
\begin{equation*}
\left|
d_k-\frac{\Delta\delta_k\alpha}{2}
\right|\omega_k
\leq
d_k\omega_k
=
b_k\delta_k.
\end{equation*}
Lemma~\ref{lm:shifted_energy_recursion} now gives
$$
\widetilde{\mathcal E}(\bs z_{k+1};\mu_{k+1})
\leq
r_k\widetilde{\mathcal E}(\bs z_k;\mu_k),
\qquad
k\geq K.
$$
Hence
$$
\widetilde{\mathcal E}(\bs z_{k+1};\mu_{k+1})
\leq
\widetilde{\mathcal E}(\bs z_K;\mu_K)
\prod_{i=K}^{k}r_i.
$$
The finitely many initial terms can be absorbed into the constant. Specifically, let
$$
C
:=
1+
\max_{0\leq j\leq K}
\left\{
\widetilde{\mathcal E}(\bs z_{j+1};\mu_{j+1})
\left(\prod_{i=0}^{j}r_i\right)^{-1}
\right\},
$$
Then, for all $k\geq0$,
$$
\widetilde{\mathcal E}(\bs z_{k+1};\mu_{k+1})
\leq
C\prod_{i=0}^{k}r_i.
$$

\medskip
\noindent
\textbf{Step 4: Asymptotic rate.}
Since $c_k\to2$ and $\Delta\delta_k\to0$,
$$
r_k
\to
r_\ast
:=
\left(1+2\sqrt{2/\kappa}\right)^{-1}.
$$
Therefore,
$$
\begin{aligned}
\limsup_{k\to\infty}
\widetilde{\mathcal E}(\bs z_{k+1};\mu_{k+1})^{1/k}
&\leq
\limsup_{k\to\infty}
\left(
C\prod_{i=0}^{k}r_i
\right)^{1/k}
=
r_\ast.
\end{aligned}
$$
Hence $\widetilde{\mathcal E}(\bs z_k;\mu_k)$ converges $R$-linearly with asymptotic rate $r_\ast$. \RV{The bound in the theorem then gives the same asymptotic rate for $\|y_k-x^\star\|^2$.} $\Box$
\end{proof}

Several remarks are in order.

\begin{remark}
The parameters $(\delta_k,\mu_k,\omega_k,b_k)$, and hence the coercivity constant $c_k$, are introduced only for the analysis and do not enter the algorithm. HNAG$^{++}$ always has the global rate
$$
(1+\sqrt{2/\kappa})^{-1}
=
1-\sqrt{2/\kappa}+O(\kappa^{-1}),
$$
while the LAS analysis improves this to the asymptotic rate $1-2\sqrt{2/\kappa}+O(\kappa^{-1})$. The index $K$ is a conservative threshold in the proof; the asymptotic regime may begin earlier, for example, when $\Delta_f(x_{k+1},x^\star)$ is already a higher-order perturbation; see Proposition~\ref{prop:full-shift}.
\end{remark}

\begin{remark}
The asymptotic factor
$$
1-2\sqrt{2/\kappa}+O(\kappa^{-1})
$$
matches the best known asymptotic rate for $\mathcal C^2$ functions~\cite{vanScoy2025fastest_firstorder_twiceC2}. Our proof uses the ODE/Lyapunov structure and differs from the argument in~\cite{vanScoy2025fastest_firstorder_twiceC2}. Local regularity enters through the Bregman asymmetry ratio
$$
\frac{|\Delta_f(x,x^\star)|}{D_f(x,x^\star)}.
$$
This extends the analysis beyond the $\mathcal C^2$ setting to the larger class $\mathcal S_{\mu,L}^{\rm LAS}(x^\star)$; see Proposition~\ref{prop:directional-C2-LAS} for a sufficient directional condition.
\end{remark}

\begin{remark}\label{rm:NAG}
The same argument applies to lower-coercivity discretizations with $c_L=1$. In particular, HNAG with $\alpha=\sqrt{\rho}$ and NAG in the equivalent HNAG form attain the asymptotic leading rate
$$
1-2\sqrt{\rho}+O(\rho)
$$
for $f\in\mathcal S_{\mu,L}^{\rm LAS}(x^\star)$. When $\Delta_f(x_k,x^\star)=0$, this recovers the sharp quadratic NAG factor, consistent with~\cite{kim2018adaptive_restart_OGM}.
\end{remark}



\section{Numerical Experiments}

We evaluate the proposed HNAG$^+$ and HNAG$^{++}$ methods on convex optimization problems. We compare HNAG$^{++}$ with Nesterov's accelerated gradient (NAG)~\cite{nesterov1983method,nesterov2003introductory}, Triple Momentum (TM)~\cite{van2017fastest}, and C$^2$-Momentum (C2M)~\cite{vanScoy2025fastest_firstorder_twiceC2}. For each example, the random seed is fixed across all methods.

All internal variables are aligned, for example, $x_0=y_0$ for HNAG methods. We stop when
$$
\|\nabla f(x_k)\|\le 10^{-8}\|\nabla f(x_0)\|.
$$

Unless stated otherwise, convergence is measured by the squared $\ell^2$-error of the iterate that is not used for gradient evaluation. For HNAG$^+$ and HNAG$^{++}$, the Lyapunov function controls both the $x$- and $y$-components, but the coercivity in the $x$-component is not strictly positive, so we report $\|y_k-x^\star\|^2$. For NAG, TM, and C2M, the gradient is evaluated at the variable denoted by $y_k$, and the provable linear convergence is stated for the other iterate sequence, denoted by $x_k$. Accordingly, we report $\|x_k-x^\star\|^2$ for these methods. This distinction is only notational: in all cases, we plot the squared $\ell^2$-error of the iterate complementary to the one used for gradient evaluation, \RV{except in the final perturbed LAS example, where we plot the normalized Lyapunov quantity $\mathcal L_k/\mathcal L_0$}.

All plots are shown on a semilog scale, so that geometric decay $r^k$ appears as a straight line and the slope reflects the rate. A steeper slope indicates faster linear convergence.

When the reference solution $x^\star$ is not available, we approximate it by running NAG until $\|\nabla f(x_k)\|\le 10^{-8}\|\nabla f(x_0)\|$. This is sufficient because strong convexity gives $\|x_k-x^\star\|^2\leq\mu^{-2}\|\nabla f(x_k)\|^2$.

\subsection{Two-dimensional Poisson problem}
We solve the two-dimensional Poisson problem
$$
-\Delta u=b \quad \text{in }\Omega,\qquad u=0 \quad \text{on }\partial\Omega,
$$
on the unit square \(\Omega=[0,1]^2\), using linear finite elements on a uniform triangulation \(\mathcal T_h\) with mesh size $h$. The stiffness matrix \(A\) is assembled with the \texttt{iFEM} package~\cite{chen:2008ifem}, which yields the quadratic objective
$$
f(x)=\tfrac12(x-x^\star)^\top A(x-x^\star),
\qquad x^\star=0.
$$
The initial iterate \(x_0\) is drawn componentwise from \(\mathrm{Unif}(0,1)\). 

We take $h=1/160,\ 1/320,\ 1/640,\ 1/1280$. For each experiment, we run all methods five times and report the average runtime. 
For each mesh size \(h\), 
the eigenvalues of the stiffness matrix $A$ have closed-form expressions
$$
\lambda_{k,l} = 4\left(\sin^2\frac{k\pi h}{2}+\sin^2\frac{l\pi h}{2}\right),
\qquad k,l=1,\ldots,1/h-1.
$$
So 
$$
\begin{aligned}
\mu &= \lambda_{\min}=\lambda_{1,1}=8\sin^2\frac{\pi h}{2}\approx 2\pi^2 h^2, \quad \text{ and}\\
L &= \lambda_{\max}=\lambda_{1/h-1,1/h-1}=8\cos^2\frac{\pi h}{2}\approx 8.
\end{aligned}
$$
Then, the condition number \(\kappa(A)\) is
$$
\kappa(A) = \frac{\lambda_{\max}}{\lambda_{\min}}=\cot^2\frac{\pi h}{2}\approx \frac{4}{\pi^2 h^2}=\mathcal{O}(h^{-2})=\mathcal{O}(N)
$$
where \(N\) is the problem dimension. Thus halving \(h\) increases \(\kappa\) by a factor of four and doubles the iteration count for accelerated methods with \(\sqrt{\kappa}\) dependence. This behavior is confirmed by the results in Table~\ref{tab:compare_methods_Laplace2D_square}.


\begin{table}[htbp]
\centering
\renewcommand{\arraystretch}{1.25}
\resizebox{0.98\textwidth}{!}{
\begin{tabular}{lcccccccc}
\toprule
$N$      
& \multicolumn{2}{c}{25,281}
& \multicolumn{2}{c}{101,761}
& \multicolumn{2}{c}{408,321}
& \multicolumn{2}{c}{1,635,841} \\
\cmidrule(lr){2-3}\cmidrule(lr){4-5}\cmidrule(lr){6-7}\cmidrule(lr){8-9}
$\kappa$
& \multicolumn{2}{c}{$1.04\times10^4$}
& \multicolumn{2}{c}{$4.16\times10^4$}
& \multicolumn{2}{c}{$1.67\times10^5$}
& \multicolumn{2}{c}{$6.66\times10^5$} \\
\midrule
Method
& Iter & Time (s)
& Iter & Time (s)
& Iter & Time (s)
& Iter & Time (s) \\
\midrule
HNAG$^{++}$ 
& 916 & 0.10
& 1,619 & 0.62
& 2,879 & 4.56
& 5,049 & 35.45\\
HNAG$^+$  
& 1,490 & 0.14
& 2,859 & 1.05
& 5,578 & 8.86
& 11,178 & 79.89\\
TM     
& 1,490 & 0.15
& 2,859 & 1.12
& 5,578 & 9.42
& 11,178 & 81.76\\
NAG    
& 1,282 & 0.13
& 2,276 & 0.87
& 4,016 & 7.39
& 7,085 & 58.28\\
C2M    
& 1,065 & 0.13
& 2,056 & 0.85
& 4,006 & 7.32
& 7,971 & 63.08\\
\bottomrule
\end{tabular}}
\caption{Performance comparison on the two-dimensional Poisson problem.}
\label{tab:compare_methods_Laplace2D_square}
\end{table}

%

\begin{figure}[htbp]
    \centering
    \begin{minipage}{0.48\linewidth}
        \centering
\includegraphics[width=\linewidth]{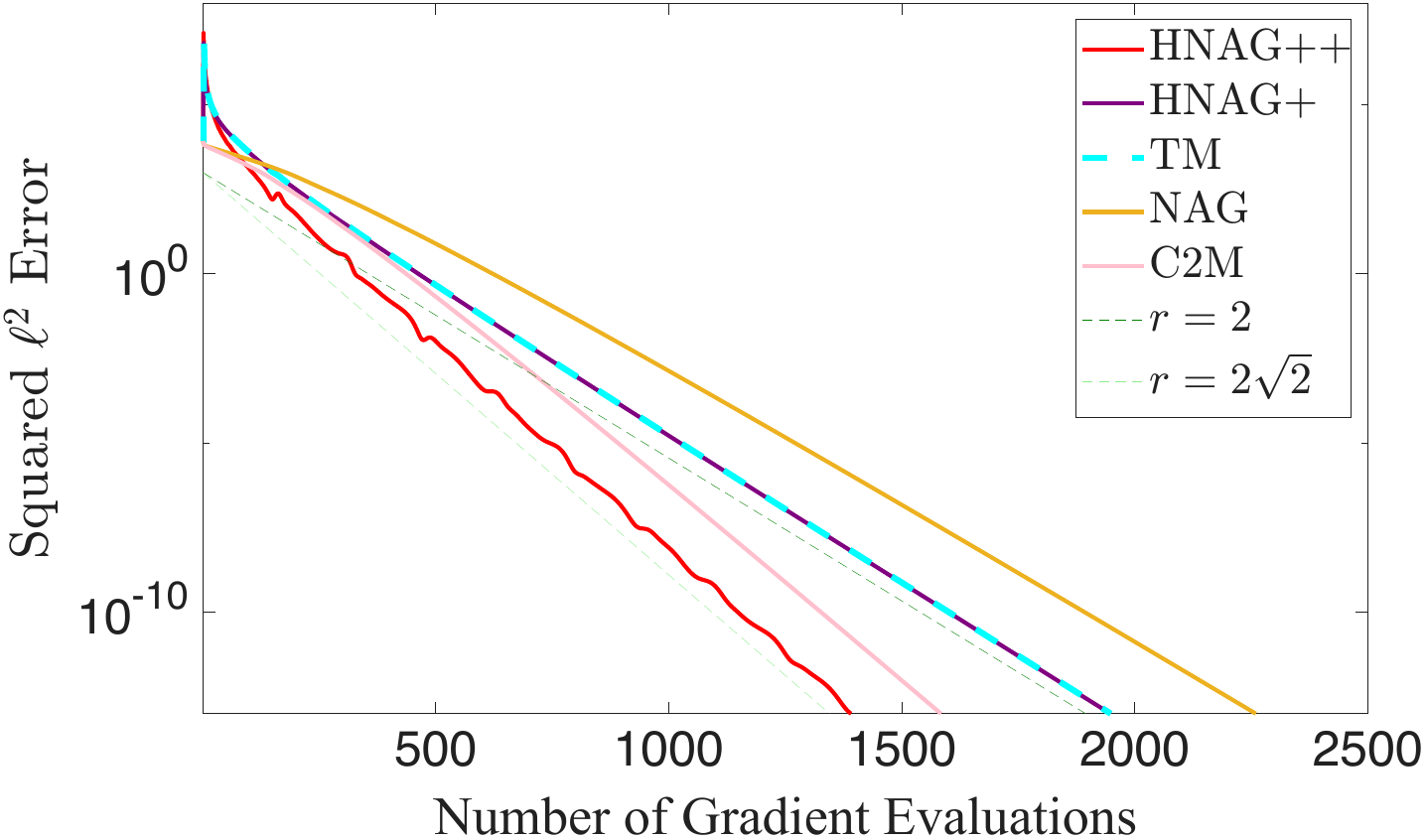}
\caption{2D linear Laplacian problem with  $h=1/160$, $n=25,281$, and $\kappa = 1.04\times10^4$.} 
\label{fig:laplace2d}
    \end{minipage}
    \hfill
    \begin{minipage}{0.48\linewidth}
        \centering
\includegraphics[width=0.98\linewidth]{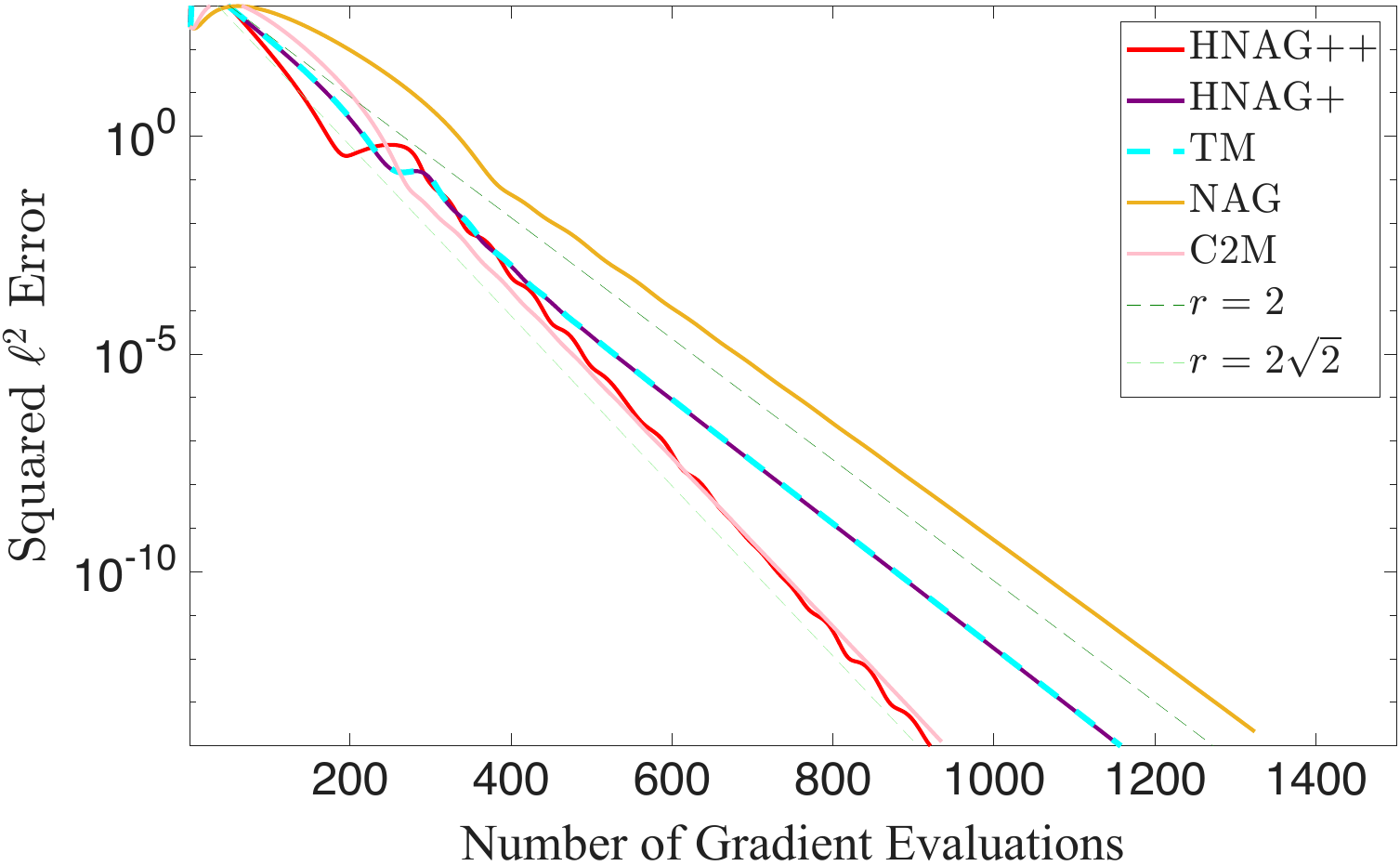}
        \caption{Logistic regression problem \eqref{eq:log-reg} with $\lambda=0.1$, $d=1000$, and $m=50$.}
        \label{fig:lor}
    \end{minipage}
\end{figure}

{HNAG$^{++}$ performs best among the methods tested. In particular, it reaches the same accuracy as NAG while reducing the computation time by about \(30\%\).} As shown in Fig.~\ref{fig:laplace2d}, the error of HNAG$^{++}$ decays linearly at the predicted rate \((1+2\sqrt{2/\kappa})^{-1}\). The other methods also show linear convergence: C2M follows its asymptotic rate \(1-2\sqrt{2/\kappa}\); HNAG$^+$ and TM attain the rate \(1-2/\sqrt{\kappa}\) and behave almost identically; and NAG follows the rate \(1-2/\sqrt{\kappa}\) proved in \cite{kim2018adaptive_restart_OGM}; see also Remark~\ref{rm:NAG}.

%
%

\subsection{Regularized logistic regression}

We consider the regularized logistic regression problem
\begin{equation}\label{eq:log-reg}
    f(x)
    = \sum_{i=1}^m \log\bigl(1+\exp(-b_i a_i^{\top} x)\bigr)
      + \frac{\lambda}{2}\|x\|^2,
\end{equation}
where $(a_i,b_i)\in\mathbb{R}^d\times\{-1,1\}$.  
The function $f$ is $\mu$-strongly convex with $\mu=\lambda$ and $L$-smooth with
$$
L = \frac{1}{4}\lambda_{\max}\!\Bigl(\sum_{i=1}^m a_i a_i^{\top}\Bigr) + \lambda .
$$
{Moreover, $f\in \mathcal C^2$, and hence $f$ is LAS.}

In the experiments, the data $a_i$ and $b_i$ are generated from normal and Bernoulli distributions, respectively, with $\lambda=0.1$, $d=1000$, and $m=50$.

Figure~\ref{fig:lor} reports the decay of $\ell^2$ errors.  
All methods exhibit accelerated linear convergence.  
HNAG$^{++}$ and C2M achieve the fastest convergence, matching the rate $1-2\sqrt{2/\kappa}$ asymptotically, while HNAG$^+$ and TM follow the rate $1-2/\sqrt{\kappa}$.  
{NAG converges more slowly in the pre-asymptotic regime, although its asymptotic rate is still \(1-2/\sqrt\kappa\); see Remark~\ref{rm:NAG}}.

\vspace{1.6em}

\subsection{\RV{Piecewise quadratic function with smooth perturbation}}

We test HNAG$^{++}$ on a LAS function with nonzero Bregman asymmetry. Let
$$
\mu=\lambda_1<\lambda_2<\cdots<\lambda_{d+1}=L,
$$
and define
\begin{equation*}
\phi_i(t):=
\begin{cases}
\lambda_i t^2, & t<0,\\[1mm]
\lambda_{i+1} t^2, & t\ge 0,
\end{cases}
\qquad i=1,\ldots,d.
\end{equation*}
Consider
\begin{equation}\label{eq:pwquad2}
f(x)=\frac12\sum_{i=1}^d \phi_i(x_i)+ \varepsilon \sum_{i=1}^d \sin^2(x_i), \qquad 0<\varepsilon<\mu/4.
\end{equation}
Then $f\in\mathcal C^1(\mathbb R^d)$ but $f\notin\mathcal C^2(\mathbb R^d)$, because each $\phi_i''$ jumps at $0$. The piecewise quadratic part has diagonal Hessian with entries in $[\mu,L]$ on each orthant. The perturbation has second derivative bounded in absolute value by $2\varepsilon$. Hence $f\in\mathcal S_{\mu-2\varepsilon,L+2\varepsilon}$ with minimizer $x^\star=0$. It also has nonzero Bregman asymmetry, while
$$
|\Delta_f(x,0)|=O(\|x\|^4),
\qquad
D_f(x,0)\geq\frac{\mu-2\varepsilon}{2}\|x\|^2.
$$
Hence
$$
\frac{|\Delta_f(x,0)|}{D_f(x,0)}
=
O(\|x\|^2)
\to0
\qquad\text{as }x\to0.
$$
Thus this example directly tests the LAS perturbation analysis in Theorem~\ref{thm:main-thm-C21}.

We set $d=1000$, $\mu=0.005$, $L=10^4$, and $\varepsilon=0.01\mu$. Then
$$
\mu_{\rm eff}=\mu-2\varepsilon=0.98\mu,
\qquad
L_{\rm eff}=L+2\varepsilon=L+0.02\mu,
\qquad
\kappa_{\rm eff}=\frac{L_{\rm eff}}{\mu_{\rm eff}}.
$$
The parameters are
$$
\lambda_i=\mu+\frac{(i-1)(L-\mu)}{d},\qquad i=1,\ldots,d+1.
$$
The initial iterate is drawn componentwise from $\mathrm{Unif}(0,1)$. 

In this example, we report both the normalized Lyapunov quantity and the squared iterate error. The Lyapunov plot shows the quantity controlled by the proof, while the squared-error plot shows the practical convergence of the iterates. For HNAG$^{++}$, HNAG$^+$, and TM, we use $\mathcal L_k=\mathcal E(\bs z_k;\mu_{\rm eff})$, with $\mathcal E$ defined in \eqref{eq:lya-muhat}, normalized by its initial value. For NAG and C2M, we plot the squared $\ell^2$-error of the corresponding convergent auxiliary sequence, normalized by its initial value.

\begin{figure}
  \centering
  \begin{minipage}{0.48\linewidth}
    \centering
    \includegraphics[width=\linewidth]{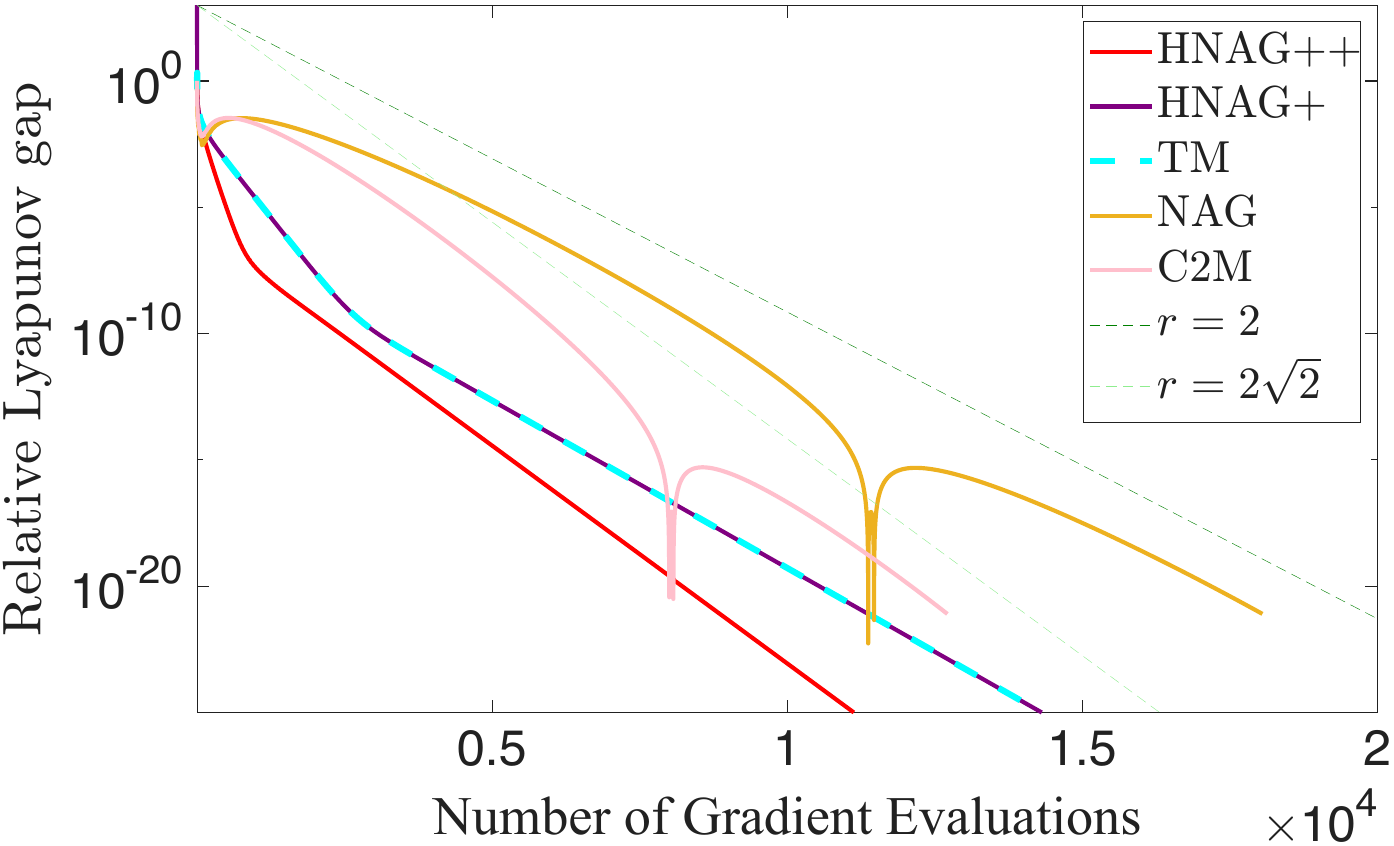}
(a) Normalized Lyapunov quantity.
  \end{minipage}
  \hfill
  \begin{minipage}{0.48\linewidth}
    \centering
    \includegraphics[width=\linewidth]{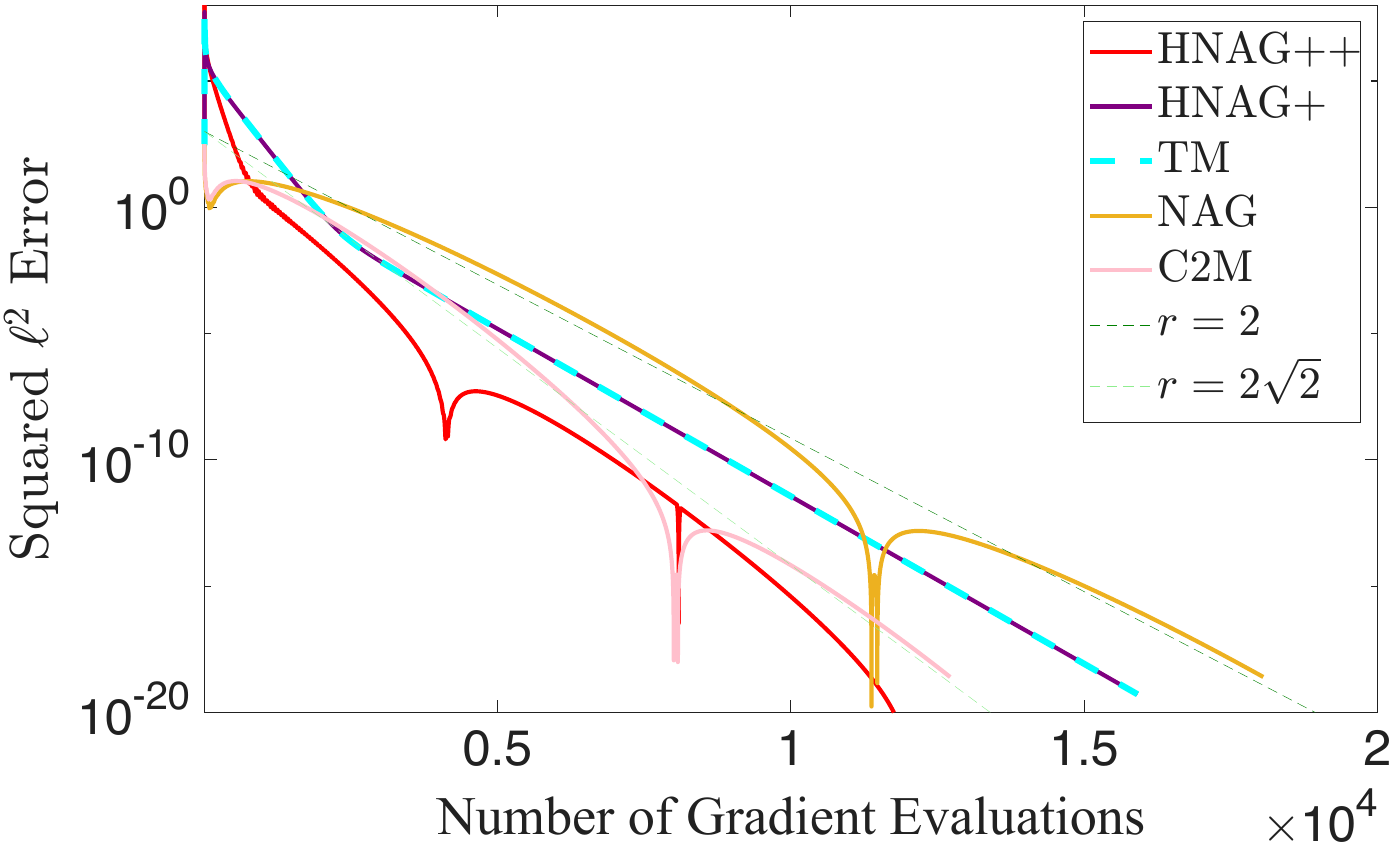}
(b) Squared iterate error.
  \end{minipage}
  \caption{Perturbed LAS example \eqref{eq:pwquad2} with $d=1000$, $\mu=0.005$, $L=10^4$, and $\varepsilon=0.01\mu$.}
  \label{fig:pwnonquad}
\end{figure}

Figure~\ref{fig:pwnonquad} shows the decay. HNAG$^{++}$ has the fastest Lyapunov decay, close to the rate $1-2\sqrt{2/\kappa_{\rm eff}}$. HNAG$^+$ and TM follow the rate $1-2/\sqrt{\kappa_{\rm eff}}$. NAG and C2M also converge, but show visible oscillations. Such oscillations also appear in the squared-error curve of HNAG$^{++}$, as accelerated gradient methods generally do not guarantee monotone decay of the objective value or iterate error. The asymptotic rates of NAG and C2M remain $1-2/\sqrt{\kappa_{\rm eff}}$ and $1-2\sqrt{2/\kappa_{\rm eff}}$, respectively; see Remark~\ref{rm:NAG}.

This example confirms that HNAG$^{++}$ remains effective when the Bregman asymmetry is nonzero but asymptotically negligible.

\FloatBarrier

\medskip
{\small\noindent\textbf{Data Availability Statement}\par
No external datasets were used for the research described in the article. The numerical data were generated synthetically as described in Section~6.\par}

\section*{Acknowledgments}
The authors thank the reviewers for their careful reading and constructive comments, which significantly improved the results and the presentation of the paper. This work was partially supported by the National Science Foundation under grant DMS--2309785.

\bibliographystyle{spmpsci}
\bibliography{reference,Optimization}
\end{document}